\documentclass[11pt,a4paper]{article}

\usepackage{a4wide}
\usepackage{graphicx}
\usepackage{subcaption}
\usepackage{amssymb,amsmath,amsthm,mathrsfs,xfrac}
\usepackage{enumerate,color}
\usepackage{hyperref}
\usepackage{array}
\usepackage{amsfonts}
\usepackage{psfrag}
\usepackage{pgf,tikz}
\usetikzlibrary{arrows}
\usepackage{mathtools}
\usepackage[title]{appendix}
\usepackage{ulem}
\usepackage{cancel}

\newtheorem{theorem}{Theorem}
\newtheorem{assumption}{Assumption}
\newtheorem{cor}[theorem]{Corollary}
\newtheorem{definition}{Definition}
\newtheorem{remark}{Remark}

\newtheorem{lemma}{Lemma}
\newtheorem{proposition}{Proposition}

\numberwithin{equation}{section}
\allowdisplaybreaks[4]

\renewenvironment{proof}{\smallskip\noindent\emph{\textbf{Proof.}}%
  \hspace{1pt}}{\hspace{-5pt}{\nobreak\quad\nobreak\hfill\nobreak%
    $\square$\vspace{2pt}\par}\smallskip\goodbreak}

\newcommand{\C}[1]{\mathbf{C^{#1}}}
\newcommand{\Cc}[1]{\mathbf{C}_c^{#1}}

\newcommand{\BV}{\mathbf{BV}}
\renewcommand{\L}[1]{\mathbf{L^#1}}
\newcommand{\Lloc}[1]{\mathbf{L^{#1}_{loc}}}

\newcommand{\modulo}[1]{{\left|#1\right|}}
\newcommand{\norma}[1]{{\left\|#1\right\|}}

\newcommand{\R}{\mathbb{R}}

\newcommand{\N}{\mathbb{N}}

\newcommand{\Z}{\mathbb{Z}}

\renewcommand{\epsilon}{\varepsilon}
\renewcommand{\phi}{\varphi}
\renewcommand{\theta}{\vartheta}

\newcommand{\tv}{\mathinner{\rm TV}}

\renewcommand{\d}[1]{\mathinner{\mathrm{d}{#1}}}

\newcommand{\del}{\partial}
\newcommand{\be}{\begin{equation}}
\newcommand{\ee}{\end{equation}}

\definecolor{ffqqqq}{rgb}{1.,0.,0.}
\definecolor{uuuuuu}{rgb}{0.26666666666666666,0.26666666666666666,0.26666666666666666}

\graphicspath{ {./images/} }

\DeclareMathOperator{\sgn}{sgn}

\makeatletter
\let\@fnsymbol\@arabic
\makeatother

\title{ 
Non-local traffic flow models with time delay: \\
well-posedness and numerical approximation}

\author{
\textsc{Ilaria Ciaramaglia\footnotemark[1]}
\and
\textsc{Paola Goatin\footnotemark[1]}
\and
\textsc{Gabriella Puppo\footnotemark[2]}
}

\date{\today}

\begin{document}
\maketitle

\footnotetext[1]{ Universit\'e C\^ote d'Azur, Inria, CNRS, LJAD, 2004 route des
 Lucioles - BP 93, 06902 Sophia Antipolis Cedex, France. E-mail:
 \texttt{\{ilaria.ciaramaglia, paola.goatin\}@inria.fr}}
 
 \footnotetext[2]{Dipartimento di Matematica - Sapienza, Università di Roma; P.le Aldo Moro, 5 - 00185
Roma, Italy. E-mail: \texttt{gabriella.puppo@uniroma1.it}}

\begin{abstract}
We prove the well-posedness of weak entropy solutions of a scalar non-local traffic flow model with time delay. Existence is obtained by convergence of finite volume approximate solutions constructed by Lax-Friedrich and Hilliges-Weidlich schemes, while the $\L1$ stability with respect to the initial data and the delay parameter relies on a Kru{\v{z}}kov-type doubling of variable technique. \\
Numerical tests are provided to illustrate the efficiency of the proposed schemes, as well as the solution dependence on the delay and look-ahead parameters.

\medskip

  \noindent\textit{2020~Mathematics Subject Classification:} 35L65, 35L03,
  65M12, 76A30.

  \medskip

  
 \end{abstract}
 
\section{Introduction}

Non-local traffic flow models~\cite{BlandinGoatin2016, Chiarello2020, ChiarelloGoatin2018, 2018Gottlich} have been recently introduced in the literature as extensions of classical macroscopic descriptions, to account for short / long range interactions among vehicles and overcome some drawbacks of local dynamic descriptions, such as the display of infinite acceleration. 
They are based on the assumption that drivers adapt their speed to a weighted mean of the downstream traffic density or velocity, thus resulting in integral dependencies of the flux function on the unknown. 
It has been shown that such models inherit the main properties of the original (local) equations, such as the maximum principle and finite, anisotropic, propagation speed of information, while gaining in regularity, which in turn provides other desirable features, such as the previously mentioned finite acceleration. In particular, the non-local interaction range can be used to model connected autonomous vehicle dynamics. On the purely mathematical side, the gain in regularity provides well-posedness results for problems whose local counterpart misses analytical results, as in the case of multi-class models~\cite{ChiarelloGoatin2019}. 

The effect of delay, accounting for drivers' reaction time,  on traffic flow dynamics has been investigated in literature both using microscopic and macroscopic models \cite{Burger2018, Gottlich2021}, showing that local delayed models are able to reproduce stop-and-go waves at a macroscopic level. The delay has been included in several forms such as in the negative part of the diffusion term in a diffusively corrected Lighthill-Whitham-Richard model \cite{Nelson2000,Tordeux2018}, or in the relaxation terms in second order models \cite{Ngoduy2014}. Modeling properly this phenomenon is crucial for developing techniques aimed at reduce traffic instabilities. 

More recently, Keimer and Pflug~\cite{KeimerPflug2019} proved the existence and uniqueness of solutions for a scalar non-local traffic model with time delay, whose local counterpart is still not fully understood~\cite{Burger2018, SIPAHI2007, Tordeux2018}. The authors' approach of choosing a priori a time delay parameter representing the reaction time seems to be more suitable for applications and data-driven modeling, with respect to the space-time non-locality considered in \cite{Du2023}.
They also investigated the behavior of the solution for the delay approaching zero, showing that it converges in the $\L 1$-norm to the solution of the associated non-delayed model. Unlike the well-posedness result, this convergence holds only on a sufficiently small time horizon.

Entropy solutions of classical (local and non-delayed) scalar conservation laws are known to satisfy a strict maximum principle. In particular, in traffic flow applications, such property ensures that the vehicles density never exceeds the maximum capacity of the considered road. While non-local models without time delay satisfy the maximum principle~\cite{BlandinGoatin2016, ChiarelloGoatin2018}, the local \cite{Gottlich2021} and non-local delayed models \cite{KeimerPflug2019} introduced fail to ensure that the traffic density stays bounded by the maximal one, also for small times.

To address the above mentioned drawbacks, in this paper we propose a non-local scalar model with time delay for which a positivity property and the maximum principle are fulfilled. We also succeed in providing uniform $\BV$ bounds, well-posedness results and convergence to the non-delayed associated model, without any restriction on the time horizon, unlike \cite{KeimerPflug2019}.

\subsection{Modeling}
We fix a constant $\tau >0$ representing the human reaction time, and we consider the following non-local traffic flow model with time delay
\begin{equation}\label{delay_bis}
\partial_t \rho(t,x)+\partial_x\big(\rho(t,x)f(\rho(t,x))v((\rho\ast \omega)(t-\tau,x))\big)=0,
\end{equation}
where $\rho:\R^+ \times \R\to [0,R]$ is the vehicle density,
$v: [0,R] \to [0,V]$ is the mean traffic speed and $\omega:[0,L]\to \R^+$ is a convolution kernel. The positive constants $R$, $V$ and $L$ are respectively the maximal traffic density, the maximal speed and the look-ahead distance of drivers.

We remark that, if the function $f:[0,R]\to [0,1]$ is chosen to be constantly equal to $1$, then the equation~\eqref{delay_bis} coincides with the model considered in~\cite{KeimerPflug2019}, that is
\begin{equation}\label{delay}
\partial_t \rho(t,x)+\partial_x\big(\rho(t,x)v((\rho\ast \omega)(t-\tau,x))\big)=0.
\end{equation}
Here, $f$ is introduced to play the role of a saturation function (see e.g.~\cite{ChiarelloGoatin2018,2018Gottlich,SopasakisKatsoulakis2006}), to serve as an indicator of the free space available on the road, whose properties are detailed later.

Due to the delayed time dependence, equation~\eqref{delay_bis} needs to be coupled with an initial condition defined on the interval $[-\tau,0]$.
In certain modelling applications it might only be possible to gather the traffic data at a given initial time $t=0$.
Thus, we couple \eqref{delay_bis} with an initial condition obtained as a constant extension of the initial datum $\rho(0,x)=\rho^0(x)\in [0,R]$, meaning that we assume
\begin{equation}\label{eq:initial datum}
\rho(t,x)=\rho^0(x),\qquad \forall (t,x)\in [-\tau,0]\times\R.
\end{equation}
For this particular choice of past-time data, the problem can be read as a classical Cauchy problem. 

Unlike similar non-local equations, we underline again that the model \eqref{delay_bis} is characterized by the presence of a delay in time, which makes general theoretical results \cite{ChiarelloGoatin2018} inapplicable as such. To obtain the well-posedness of the Cauchy problem~\eqref{delay_bis},~\eqref{eq:initial datum}, we require the following assumptions.
\begin{assumption}\label{hp}
    For any $T>0$, it holds:
    \begin{itemize}
        \item $v\in\C2([0,R];[0,V])$, $v'\leq 0$, $v(0)=V$ and $v(R)=0$; 
        \item $f\in\C1 ([0,R];[0,1])$, $f'(\rho)\leq 0$, $f(0)=1$ and $f(R)=0$;
        \item $\omega\in \C1 ([0,L];\R^+)$,  $\omega'(x)\leq 0$ and 
        \begin{equation}\label{J0}
        \int_0^L\omega(s)\d s=J_0=1.
        \end{equation}
    \end{itemize}
    We extend the kernel $\omega(x)=0$ for $x>L$, and the saturation function $f(\rho)=0$ for $\rho>R$, and $f(\rho)=1$ for $\rho<0$.
\end{assumption}

\begin{remark}\label{hpremark}
    The assumption on the velocity function $v$ is classical, but it can probably be weakened in some of the following results. For instance, in~\cite{KeimerPflug2019} the weaker condition $v\in W^{1,+\infty}_{loc}(\mathbb{R^+})$ is taken, which is sufficient to prove existence and uniqueness of solutions to~\eqref{delay}, even if just on a sufficiently small time horizon. 
    Also, the less restrictive
    hypothesis on the interaction kernel $\omega\in W^{1,\infty}(\R,\R^+)$ is used. 
    We need the regularity and monotonicity of $v$ and $\omega$ to derive the $\BV$ estimate for the approximate solutions. Moreover,  it is physically reasonable to suppose that $\omega$ is a non-increasing function of the distance, since this means that the drivers assign greater importance to closer vehicles.
\end{remark}

\begin{remark}
        For simplicity, we assume $\int_0^L\omega(s)\d s=1$, thus avoiding to extend $v$ beyond $R$. 
        Anyway, the following results can be generalized to any positive interaction strength $J_0>0$.
\end{remark}

Under the above hypothesis, we prove the existence and stability of weak entropy solutions of~\eqref{delay_bis},~\eqref{eq:initial datum} defined as follows~\cite{Kruzkov}.

\begin{definition}[Entropy weak solution]\label{entropy}
Given $\rho^0\in \L1(\R;[0,R])$, a function  $\rho\in \L 1([0,T]\times\R;\R)$ is an entropy weak solution of the Cauchy problem \eqref{delay_bis}-\eqref{eq:initial datum} if for every  test function $\phi\in\Cc 1([0,T)\times\R;\R^+)$ and $\kappa\in\R$
\begin{align*}
    \int_0^T\int_\R&\big(\modulo{\rho-\kappa}\del_t\phi+\sgn(\rho-\kappa)(\rho f(\rho)-\kappa f(\kappa))v((\rho\ast\omega)(t-\tau,x))\del_x\phi\\
    &-\sgn(\rho-\kappa)\kappa f(\kappa)\del_x v((\rho\ast\omega)(t-\tau,x))\phi\big)\mathrm{d}x\mathrm{d}t\\
    &+\int_\R\modulo{\rho^0(x)-\kappa}\phi(0,x)\mathrm{d}x\geq 0\,.
\end{align*}
\end{definition}

Observe that the introduction of the saturation function $f$ guarantees that a weak maximum principle holds, unlike the case of equation~\eqref{delay}, see~\cite[Corollary 4.4]{KeimerPflug2019}. Indeed, since the constants $0$ and $R$ are solutions of \eqref{delay_bis}, by comparison we have $0\leq\rho(t,x)\leq R$ for all $(t,x)\in \R^+\times\R$. 
On the other hand, $F(\rho):=\rho f(\rho)$ being nonlinear, it prevents the use of the approach based on characteristics and fixed point theorem exploited in~\cite{KeimerPflug2019}, which provides uniqueness of weak solutions without enforcing entropy conditions. Therefore, in this work we need to refer to entropy weak solutions as in Definition~\ref{entropy}.

\subsection{Structure of the contribution}
The rest of the paper is organized as follows. Sections~\ref{sec:FVapprox} and ~\ref{convergencemono} are devoted to the construction of two types of finite volume approximate solutions and the derivation of the corresponding uniform compactness estimates. Their convergence towards a global entropy weak solution is proved in Section~\ref{sec:convergence}, where we provide the existence in Theorem~\ref{E1}, together with with the maximum principle, the bound on the total variation and the $\L1$ stability in time,
and the stability in Theorem~\ref{stabilityteo}. As a consequence, we show in Corollary \ref{stabilitycor} the $\L 1$ convergence to the associated non-delayed model as the time delay parameter $\tau$ goes to zero. We stress that our results hold on every finite time horizon, which is one of the improvements with respect to the results on the non-local model with time delay proposed in~\cite{KeimerPflug2019}.
A numerical study of the solution properties is deferred to Section~\ref{sec:num}, with a particular focus on the capability of the saturation function $f$ in guaranteeing the maximum principle, and on the influence of the delay on the increase of the total variation of the solution. 
Some technical details about the proofs are reported in Appendix~\ref{sec:app}.

\section{Finite volume approximations}
\label{sec:FVapprox}
We take a space step $\Delta x$ such that for the size of the kernel support we have $L=N\Delta x$ for some $N\in\mathbb{N}$ holds, and a time step $\Delta t$ such that the time delay parameter is $\tau=h\Delta t$ for some $h\in\N$. 
We discretize \eqref{delay} on a fixed grid made up of the cell centers $x_j=(j-\frac{1}{2})\Delta x$, the cell interfaces $x_{j+\frac{1}{2}}=j\Delta x$ for $j\in\mathbb{Z}$, and the time mesh $t^n=n\Delta t$, $n\in\N$.
We want to build a finite volume approximate solution, denoted as $\rho^{\Delta x}(t,x)=\rho^n_j$ for $(t,x)\in[t_n,t_{n+1})\times [x_{j-\frac{1}{2}},x_{j+\frac{1}{2}})$.
In order to do this, first we approximate the initial datum $\rho^0$ with the piecewise constant function
$$
\rho^0_j=\frac{1}{\Delta x}\int_{x_{j-\frac{1}{2}}}^{x_{j+\frac{1}{2}}} \rho^0(x)\d x, \qquad j\in\mathbb{Z},
$$
 and we set $\rho^{-n}_j=\rho^0_j$ for $j\in\mathbb{Z}$ and $n=1,\dots,h$, consistently with \eqref{eq:initial datum}.
Similarly, for the kernel, we define
$$
\omega^k:=\frac{1}{\Delta x}\int_{k\Delta x}^{(k+1)\Delta x}\omega(x)\d x, \qquad k\in\mathbb{N}_0=\N\cup\{0\},
$$
so that from \eqref{J0} we get $\Delta x\sum_{k=0}^{+\infty}\omega^k=\Delta x\sum_{k=0}^{N-1}\omega^k=J_0$, where the sum is finite since $\omega^k=0$ for $k\geq N$ sufficiently large.
Then, we denote 
$$
V^n_j:=v\left( \Delta x\sum_{k=0}^{+\infty}\omega^k\rho^n_{j+k}\right)=v\left( \Delta x\sum_{k=0}^{N-1}\omega^k\rho^n_{j+k}\right),
$$
which involves a quadrature formula to approximate the convolution term.
Thus, as in~\cite{ChiarelloGoatin2018}, we consider the following Lax-Friedrichs (LF) numerical flux
$$
(\mathcal{F}_{LF})^n_{j+\frac{1}{2}}=\frac{1}{2}\left(F(\rho^n_j)V^{n-h}_j+F(\rho^n_{j+1})V^{n-h}_{j+1}\right)-\frac{\alpha}{2}\left(\rho^n_{j+1}-\rho^n_j\right),
$$
with $\lambda=\Delta t/\Delta x$ and $F(\rho)=\rho f(\rho)$, which leads to the following numerical scheme
\begin{equation}\label{schema}
\rho^{n+1}_j=\rho^n_j+\frac{\lambda\alpha}{2}\left(\rho^n_{j+1}-2\rho^n_j+\rho^n_{j-1}\right)-\frac{\lambda}{2}\big(F(\rho^n_{j+1})V^{n-h}_{j+1}-F(\rho^n_{j-1})V^{n-h}_{j-1}\big),
\end{equation}
being $\alpha$ the numerical viscosity coefficient. 
Moreover, following~\cite{chiarello2023existence, ChiarelloGoatin2019, 2018Gottlich}, we consider the following Hilliges-Weidlich (HW) type scheme~\cite{Hilliges1995}
\begin{equation}\label{schema2}
\rho^{n+1}_j=\rho^n_j-\lambda\left(\rho^n_jf(\rho^n_{j+1})V^{n-h}_{j+1}-\rho^n_{j-1}f(\rho^n_j)V^{n-h}_j\right),
\end{equation}
with numerical flux
$$
(\mathcal{F}_{HW})^n_{j+\frac{1}{2}}=\rho^n_jf(\rho^n_{j+1})V^{n-h}_{j+1}.
$$

To prove the convergence of the corresponding approximate solutions, we report the statements for both the schemes and detailed calculations for the LF scheme, while the proofs regarding the HW scheme involve simpler calculations and are sometimes only sketched. In Section~\ref{confrontoschemi}, we propose a numerical comparison between these schemes, showing that the HW scheme is less diffusive than the widely used LF scheme.\\
We start providing some important properties of the LF scheme.

\begin{lemma}[Positivity]\label{posteo} For any $T>0$, if
\begin{equation} \label{eq:CFL1}
    \alpha\geq V\qquad\mbox{ and }\qquad\lambda\leq\frac{1}{\alpha},
\end{equation}
then the LF scheme \eqref{schema} is positivity preserving on $[0,T]\times\R$.
\end{lemma}
\begin{proof}
Let us assume that $\rho^{n-l}_j\geq 0$ for all $j\in\mathbb{Z}$ and $l=0,\dots,h$. Then, using \eqref{schema} we have
$$
\rho^{n+1}_j=(1-\lambda\alpha)\rho^n_j+\frac{\lambda}{2}\big(\alpha-f(\rho^n_{j+1})V^{n-h}_{j+1}\big)\rho^n_{j+1}+\frac{\lambda}{2}\big(\alpha+f(\rho^n_{j-1})V^{n-h}_{j-1}\big)\rho^n_{j-1}\geq 0,
$$
since all the coefficients are non-negative.
\end{proof}

\begin{lemma}[$\L 1$-bound]\label{L1boundteo}
For any $n\in\mathbb{N}$, under the CFL condition \eqref{eq:CFL1} the approximate solutions constructed by the LF scheme \eqref{schema} satisfy
$$
\norma{\rho^n}_1=\norma{\rho^0}_1,
$$
where $\norma{\rho^n}_1:=\Delta x\sum_j\modulo{\rho^n_j}$ denotes the $\L 1$-norm of $\rho^{\Delta x}(n\Delta t,\cdot)$.
\end{lemma}

\begin{proof}
Thanks to Lemma \ref{posteo}, we have
\begin{align*}
\norma{\rho^{n+1}}_1=&\Delta x\sum_j\rho^{n+1}_j=\Delta x\sum_j\rho^n_j-\lambda\Delta x\sum_j\left((\mathcal{F}_{LF})^n_{j+\frac{1}{2}}-(\mathcal{F}_{LF})^n_{j-\frac{1}{2}}\right)=\Delta x\sum_j\rho^n_j=\norma{\rho^n}_1.
\end{align*}
\end{proof}

In the sequel, we use the compact notation $\norma{\cdot}$ for $\norma{\cdot}_{\infty}$.

\begin{lemma}[$\L\infty$-bound / weak maximum principle]\label{boundteo}
If $\rho^0_j\in[0,R]$ for all $j\in\mathbb{Z}$ and the condition
\begin{equation}\label{CFLtemp}
\alpha\geq V\left(1+R\norma{f'}\right)\qquad\mbox{ and }\qquad\lambda\leq\frac{1}{\alpha}
\end{equation}
holds, then the numerical solution given by the LF scheme \eqref{schema} satisfies $\rho_j^n \in [0,R]$ for all $j\in\Z$, $
n\in\N$. 
\end{lemma}

\begin{proof}
We can rewrite the scheme \eqref{schema} in the form $\rho^{n+1}_j=H(\rho^n_{j-1},\rho^n_j,\rho^n_{j+1},\rho^{n-h}_{j-1},\dots,\rho^{n-h}_{j+N})$, being
\begin{align*}
H(\rho^n_{j-1},\rho^n_j,\rho^n_{j+1},\rho^{n-h}_{j-1},\dots,\rho^{n-h}_{j+N}):=\rho^n_j&+\frac{\lambda\alpha}{2}\left(\rho^n_{j+1}-2\rho^n_j+\rho^n_{j-1}\right)\\
&-\frac{\lambda}{2}\big(F(\rho^n_{j+1})V^{n-h}_{j+1}-F(\rho^n_{j-1})V^{n-h}_{j-1}\big).
\end{align*}
We suppose that $\rho^n_j\leq R$ for $j\in\mathbb{Z}$, and we define
\begin{eqnarray*}
\mathcal{R}^n_j&=&(\rho^n_{j-1},\rho^n_j,\rho^n_{j+1},\rho^{n-h}_{j-1},\dots,\rho^{n-h}_{j+N}),\\
\mathcal{R}^n_{max}&=&(R,R,R,\rho^{n-h}_{j-1},\dots,\rho^{n-h}_{j+N}).
\end{eqnarray*}
Since $H(\mathcal{R}^n_{max})=R$,
by applying the mean value theorem between the points $\mathcal{R}^n_j$ and $\mathcal{R}^n_{max}$, we get
\begin{eqnarray*}
\rho^{n+1}_j&=&H(\mathcal{R}^n_j)=H(\mathcal{R}^n_{max})+\langle\nabla H(\mathcal{R}_\xi),\mathcal{R}^n_j-\mathcal{R}^n_{max}\rangle\\
&=&R+\langle\nabla H(\mathcal{R}_\xi),\mathcal{R}^n_j-\mathcal{R}^n_{max}\rangle,
\end{eqnarray*}
with $\mathcal{R}_\xi=(1-\xi)\mathcal{R}^n_{max}+\xi \mathcal{R}^n_j$, for some $\xi\in[0,1]$.
It is now enough to observe that
\begin{eqnarray}
\frac{\partial H}{\partial\rho^n_{j-1}}&=&\frac{\lambda}{2}\big(\alpha+F'(\rho^n_{j-1})V^{n-h}_{j-1}\big)\geq 0,\label{derivative j-1}\\
\frac{\partial H}{\partial\rho^n_{j}}&=&1-\lambda\alpha\geq 0,\label{derivative j}\\
\frac{\partial H}{\partial\rho^n_{j+1}}&=&\frac{\lambda}{2}\big(\alpha-F'(\rho^n_{j+1})V^{n-h}_{j+1}\big)\geq 0,\label{derivative j+1}
\end{eqnarray}
and that $F'(\rho)=f(\rho)+\rho f'(\rho)$.
The inequalities \eqref{derivative j-1} and \eqref{derivative j+1} holds thanks to the hypothesis on $\alpha$, while \eqref{derivative j} is due to the CFL condition.
Thus, we conclude
\begin{align*}
\langle\nabla H(\mathcal{R}_\xi),&\mathcal{R}^n_j-\mathcal{R}^n_{max}\rangle=\sum_{k=1}^{N+5}\nabla H(\mathcal{R}_\xi)_k(\mathcal{R}^n_j-\mathcal{R}^n_{max})_k\\
=&\ \frac{\partial H}{\partial\rho^n_{j-1}}(\mathcal{R}_\xi)(\rho^n_{j-1}-R)+\frac{\partial H}{\partial\rho^n_{j}}(\mathcal{R}_\xi)(\rho^n_{j}-R)+\frac{\partial H}{\partial\rho^n_{j+1}}(\mathcal{R}_\xi)(\rho^n_{j+1}-R)\leq 0,
\end{align*}
and therefore the statement.
\end{proof}
\begin{remark}[Properties of the HW scheme]\label{HWremark}
It is easy to show that the properties which we proved in Lemma \ref{posteo} and Lemma \ref{L1boundteo} hold also for the HW scheme. Thus, under the CFL condition $\lambda\leq 1/V$ the numerical solution given by \eqref{schema2} is such that
\begin{equation*}
\rho^n_j\geq 0,\qquad\mbox{ for all } j\in\Z,~n\in\N.
\end{equation*}
Under the stronger condition 
$$
\lambda\leq\frac{1}{V\left(1+R\norma{f'}\right)},
$$
then also the weak form of the maximum principle $\rho^n_j \in [0,R]$ holds, and this can be proven exactly as in the proof of Lemma~\ref{boundteo} by defining 
$$
H(\rho^n_{j-1},\rho^n_j,\rho^n_{j+1},\rho^{n-h}_j,\dots,\rho^{n-h}_{j+N}):=\rho^n_j-\lambda\left(\rho^n_jf(\rho^n_{j+1})V^{n-h}_{j+1}-\rho^n_{j-1}f(\rho^n_j)V^{n-h}_j\right),
$$
and taking
\begin{eqnarray*}
\mathcal{R}^n_j&=&(\rho^n_{j-1},\rho^n_j,\rho^n_{j+1},\rho^{n-h}_j,\dots,\rho^{n-h}_{j+N}),\\
\mathcal{R}^n_{max}&=&(\rho^n_{j-1},R,R,\rho^{n-h}_j,\dots,\rho^{n-h}_{j+N}).
\end{eqnarray*}
\end{remark}

\section{Compactness estimates}\label{convergencemono}
In \cite{AmorimColomboTexeira,ChiarelloGoatin2018,ChiarelloGoatinRossi2019} it is shown that, for non-local equations with no delay in time, the approximate solutions constructed using the adapted LF numerical scheme have uniformly bounded total variation. In the following, we derive original global $\BV$ estimates for our delayed model~\eqref{delay_bis}, for both the LF and the HW schemes. We will use such estimates to prove the convergence of the schemes in Section~\ref{sec:convergence}. 

Assuming the stronger CFL condition
\begin{equation}\label{CFL LF}
\alpha\geq V\left(1+R\norma{f'}\right)\qquad\mbox{and}\qquad\lambda\leq\frac{1}{\alpha+V\left(1+R\norma{f'}\right)},
\end{equation}
we start with the $\BV$ estimate in space for the LF scheme. We remark that we intend $\BV\subset \L1$.

\begin{proposition}[Spatial $\BV$-bound for the LF scheme]\label{spaceBVteo}
    Let Assumption~\ref{hp} and the CFL condition \eqref{CFL LF} hold. Then, for any $\rho^0\in \BV(\R;[0,R])$, the numerical solution $\rho^{\Delta x}(t,\cdot)$ given by the LF scheme \eqref{schema} has bounded total variation for $t\in[0,T]$, uniformly in $\Delta x$, for every time horizon $T>0$.
\end{proposition}

\begin{proof}
    Let us set
    $
    \Delta^n_{j+\frac{1}{2}}:=\rho^n_{j+1}-\rho^n_j.
    $
Using the mean value theorem, we obtain
\begin{align}
    \Delta^{n+1}_{j+\frac{1}{2}}=& \ \Delta^n_{j+\frac{1}{2}}+\frac{\lambda\alpha}{2}\left(\Delta^n_{j+\frac{3}{2}}-2\Delta^n_{j+\frac{1}{2}}+\Delta^n_{j-\frac{1}{2}}\right)\nonumber\\
        &-\frac{\lambda}{2}\left[F(\rho^n_{j+2})V^{n-h}_{j+2}-F(\rho^n_{j+1})V^{n-h}_{j+1}-F(\rho^n_j)V^{n-h}_j+F(\rho^n_{j-1})V^{n-h}_{j-1}\right]\nonumber\\
        =& \ \Delta^n_{j+\frac{1}{2}}+\frac{\lambda\alpha}{2}\left(\Delta^n_{j+\frac{3}{2}}-2\Delta^n_{j+\frac{1}{2}}+\Delta^n_{j-\frac{1}{2}}\right)\nonumber\\
        &-\frac{\lambda}{2}\left[F'(\tilde{\rho}^n_{j+\frac{3}{2}})V^{n-h}_{j+2}\Delta^n_{j+\frac{3}{2}}+F(\rho^n_{j+1})\left(V^{n-h}_{j+2}-V^{n-h}_{j+1}\right)\right]\nonumber\\  
        &+\frac{\lambda}{2}\left[F'(\tilde{\rho}^n_{j-\frac{1}{2}})V^{n-h}_j\Delta^n_{j-\frac{1}{2}}+F(\rho^n_{j-1})\left(V^{n-h}_j-V^{n-h}_{j-1}\right)\right]\nonumber\\
        =& \ (1-\lambda\alpha)\Delta^n_{j+\frac{1}{2}}\nonumber\\
        &+\frac{\lambda}{2}\left(\alpha-F'(\tilde{\rho}^n_{j+\frac{3}{2}})V^{n-h}_{j+2}\right)\Delta^n_{j+\frac{3}{2}}+\frac{\lambda}{2}\left(\alpha+F'(\tilde{\rho}^n_{j-\frac{1}{2}})V^{n-h}_j\right)\Delta^n_{j-\frac{1}{2}}\nonumber\\
        &-\frac{\lambda}{2}\left[F(\rho^n_{j+1})\left(V^{n-h}_{j+1}-V^{n-h}_j\right)-F(\rho^n_{j-1})\left(V^{n-h}_j-V^{n-h}_{j-1}\right)\right]\nonumber\\
        &-\frac{\lambda}{2}F(\rho^n_{j+1})\left[\left(V^{n-h}_{j+2}-V^{n-h}_{j+1}\right)-\left(V^{n-h}_{j+1}-V^{n-h}_j\right)\right]\nonumber\\
        =& \ (1-\lambda\alpha)\Delta^n_{j+\frac{1}{2}}\nonumber\\
        &+\frac{\lambda}{2}\left(\alpha-F'(\tilde{\rho}^n_{j+\frac{3}{2}})V^{n-h}_{j+2}\right)\Delta^n_{j+\frac{3}{2}}+\frac{\lambda}{2}\left(\alpha+F'(\tilde{\rho}^n_{j-\frac{1}{2}})V^{n-h}_j\right)\Delta^n_{j-\frac{1}{2}}\nonumber\\
        &-\frac{\lambda}{2}\left(F'(\tilde{\rho}^n_{j+\frac{1}{2}})\Delta^n_{j+\frac{1}{2}}+F'(\tilde{\rho}^n_{j-\frac{1}{2}})\Delta^n_{j-\frac{1}{2}}\right)\left(V^{n-h}_j-V^{n-h}_{j-1}\right)\nonumber\\
                &-\frac{\lambda}{2}F(\rho^n_{j+1})\left[\left(V^{n-h}_{j+1}-V^{n-h}_j\right)-\left(V^{n-h}_j-V^{n-h}_{j-1}\right)\right]\nonumber\\
        &-\frac{\lambda}{2}F(\rho^n_{j+1})\left[\left(V^{n-h}_{j+2}-V^{n-h}_{j+1}\right)-\left(V^{n-h}_{j+1}-V^{n-h}_j\right)\right]\nonumber\\
        =& \left[1-\lambda\alpha-\frac{\lambda}{2}F'(\tilde{\rho}^n_{j+\frac{1}{2}})\left(V^{n-h}_j-V^{n-h}_{j-1}\right)\right]\Delta^n_{j+\frac{1}{2}}\nonumber\\
        &+\frac{\lambda}{2}\left(\alpha-F'(\tilde{\rho}^n_{j+\frac{3}{2}})V^{n-h}_{j+2}\right)\Delta^n_{j+\frac{3}{2}}\nonumber\\
        &+\frac{\lambda}{2}\left[\alpha+F'(\tilde{\rho}^n_{j-\frac{1}{2}})V^{n-h}_j-F'(\tilde{\rho}^n_{j-\frac{1}{2}})\left(V^{n-h}_j-V^{n-h}_{j-1}\right)\right]\Delta^n_{j-\frac{1}{2}}\nonumber\\
        &-\frac{\lambda}{2}F(\rho^n_{j+1})\left[\left(V^{n-h}_{j+1}-V^{n-h}_j\right)-\left(V^{n-h}_j-V^{n-h}_{j-1}\right)\right]\label{b3}\\
        &-\frac{\lambda}{2}F(\rho^n_{j+1})\underbrace{\left[\left(V^{n-h}_{j+2}-V^{n-h}_{j+1}\right)-\left(V^{n-h}_{j+1}-V^{n-h}_j\right)\right]}_{=(\ast)},\label{b2}
        \end{align}
with $\tilde{\rho}^n_{j+\frac{1}{2}}$ between $\rho^n_j$ and $\rho^n_{j+1}$ for all $j\in\mathbb{Z}$.
The term $(\ast)$ in \eqref{b2} can be estimated as
    \begin{eqnarray*}
        (\ast)&=&\left(V^{n-h}_{j+2}-V^{n-h}_{j+1}\right)-\left(V^{n-h}_{j+1}-V^{n-h}_j\right)=v'(\xi_{j+\frac{3}{2}})\Delta x\left(\sum_{k=0}^{+\infty}\omega^k\rho^{n-h}_{j+k+2}-\sum_{k=0}^{+\infty}\omega^k\rho^{n-h}_{j+k+1}\right)\\
        &&-\ v'(\xi_{j+\frac{1}{2}})\Delta x\left(\sum_{k=0}^{+\infty}\omega^k\rho^{n-h}_{j+k+1}-\sum_{k=0}^{+\infty}\omega^k\rho^{n-h}_{j+k}\right)\\
        &=&v'(\xi_{j+\frac{3}{2}})\Delta x\left(\sum_{k=1}^{+\infty}(\omega^{k-1}-\omega^k)\rho^{n-h}_{j+k+1}-\omega^0\rho^{n-h}_{j+1}\right)\\
        &&-\ v'(\xi_{j+\frac{1}{2}})\Delta x\left(\sum_{k=1}^{+\infty}(\omega^{k-1}-\omega^k)\rho^{n-h}_{j+k}-\omega^0\rho^{n-h}_j\right)\\
        &=&\left[v'(\xi_{j+\frac{3}{2}})-v'(\xi_{j+\frac{1}{2}})\right]\Delta x\left(\sum_{k=1}^{+\infty}(\omega^{k-1}-\omega^k)\rho^{n-h}_{j+k+1}-\omega^0\rho^{n-h}_{j+1}\right)\\
        &&+\ v'(\xi_{j+\frac{1}{2}})\Delta x\left(\sum_{k=1}^{+\infty}(\omega^{k-1}-\omega^k)\left(\rho^{n-h}_{j+k+1}-\rho^{n-h}_{j+k}\right)-\omega^0\left(\rho^{n-h}_{j+1}-\rho^{n-h}_j\right)\right),
    \end{eqnarray*}
    with $\xi_{j+\frac{3}{2}}$ between $\Delta x\sum_{k=0}^{+\infty}\omega^k\rho^{n-h}_{j+k+2}$ and $\Delta x\sum_{k=0}^{+\infty}\omega^k\rho^{n-h}_{j+k+1}$ and  $\xi_{j+\frac{1}{2}}$ between $\Delta x\sum_{k=0}^{+\infty}\omega^k\rho^{n-h}_{j+k+1}$ and $\Delta x\sum_{k=0}^{+\infty}\omega^k\rho^{n-h}_{j+k}$.
    Thus, we can write
    \begin{eqnarray*}
        (\ast)&=&v''(\tilde{\xi}_{j+1})\left[\xi_{j+\frac{3}{2}}-\xi_{j+\frac{1}{2}}\right]\Delta x\sum_{k=1}^{+\infty}\omega^{k-1}\Delta^{n-h}_{j+k+\frac{1}{2}}+\\
        &&+\ v'(\xi_{j+\frac{1}{2}})\Delta x\left(\sum_{k=1}^{+\infty}(\omega^{k-1}-\omega^k)\Delta^{n-h}_{j+k+\frac{1}{2}}-\omega^0\Delta^{n-h}_{j+\frac{1}{2}}\right),
    \end{eqnarray*}
    where $\tilde{\xi}_{j+1}$ is between $\xi_{j+\frac{3}{2}}$ and $\xi_{j+\frac{1}{2}}$.
    For some $\theta,\mu\in[0,1]$, we compute
    \begin{align*}
        \xi_{j+\frac{3}{2}}-\xi_{j+\frac{1}{2}}=&\ \theta\Delta x\sum_{k=0}^{+\infty}\omega^k\rho^{n-h}_{j+k+2}+(1-\theta)\Delta x\sum_{k=0}^{+\infty}\omega^k\rho^{n-h}_{j+k+1}\\
        &-\mu\Delta x\sum_{k=0}^{+\infty}\omega^k\rho^{n-h}_{j+k+1}-(1-\mu)\Delta x\sum_{k=0}^{+\infty}\omega^k\rho^{n-h}_{j+k}\\
        =&\ \Delta x\sum_{k=1}^{+\infty}\left[\theta\omega^{k-1}+(1-\theta)\omega^k-\mu\omega^k-(1-\mu)\omega^{k+1}\right]\rho^{n-h}_{j+k+1}\\
        &+\Delta x\left[(1-\theta)\omega^0\rho^{n-h}_{j+1}-\mu\omega^0\rho^{n-h}_{j+1}-(1-\mu)\left(\omega^0\rho^{n-h}_j+\omega^1\rho^{n-h}_{j+1}\right)\right].
    \end{align*}
    Taking the absolute values and using Lemma \ref{boundteo}, we get
    \begin{eqnarray}
        \modulo{\xi_{j+\frac{3}{2}}-\xi_{j+\frac{1}{2}}}&\leq&\Delta x\left[\sum_{k=1}^{+\infty}\big[\theta\omega^{k-1}+(1-\theta)\omega^k-\mu\omega^k-(1-\mu)\omega^{k+1}\big]+4\omega^0\right]R\nonumber\\
        &\leq&\Delta x\left[\sum_{k=1}^{+\infty}[\omega^{k-1}-\omega^{k+1}]+4\omega^0\right]R\leq6\Delta x\norma{\omega}R,\label{delta xi}
    \end{eqnarray}
    where we used the monotonicity of $\omega$ to get $\norma{\omega}=\omega^0$ and
    $$
    \theta\omega^{k-1}+(1-\theta)\omega^k-\mu\omega^k-(1-\mu)\omega^{k+1}\geq 0.
    $$
Therefore, we obtain the bound 
\begin{equation}
\sum_j(|\eqref{b3}|+|\eqref{b2}|)\leq\Delta t\mathcal{H}\sum_j\modulo{\Delta^{n-h}_{j+\frac{1}{2}}},\label{boundH}
\end{equation}
for $\mathcal{H}=R\norma{\omega}\left(6\norma{v''}J_0R+2\norma{v'}\right)$.   
Thus, taking the absolute values and summing on $j\in\mathbb{Z}$ in the bound of $\Delta^{n+1}_{j+\frac{1}{2}}$ at the beginning of the proof, we get
\begin{align*}
    \sum_j\modulo{\Delta^{n+1}_{j+\frac{1}{2}}}\leq& \ \sum_j\left[1-\lambda\left(\alpha+\frac{1}{2}F'(\tilde{\rho}^n_{j+\frac{1}{2}})\left(V^{n-h}_j-V^{n-h}_{j-1}\right)\right)\right] \modulo{\Delta^n_{j+\frac{1}{2}}}\\
    &+\frac{\lambda}{2}\sum_j\left[\alpha-F'(\tilde{\rho}^n_{j+\frac{3}{2}})V^{n-h}_{j+2}\right]\modulo{\Delta^n_{j+\frac{3}{2}}}\\
    &+\frac{\lambda}{2}\sum_j\left[\alpha+F'(\tilde{\rho}^n_{j-\frac{1}{2}})V^{n-h}_j-F'(\tilde{\rho}^n_{j-\frac{1}{2}})\left(V^{n-h}_j-V^{n-h}_{j-1}\right)\right]\modulo{\Delta^n_{j-\frac{1}{2}}}\\
    &+\Delta t\mathcal{H}\sum_j\modulo{\Delta ^{n-h}_{j+\frac{1}{2}}},
\end{align*}
where all the coefficients are positive thanks to the CFL condition and the hypothesis on the coefficient $\alpha$.
Due to some cancellations, this becomes
\begin{align*}
    \sum_j\modulo{\Delta^{n+1}_{j+\frac{1}{2}}}\leq&\ \sum_j\left[1+\frac{\lambda}{2}F'(\tilde{\rho}^n_{j+\frac{1}{2}})\left(V^{n-h}_{j-1}-V^{n-h}_j\right)\right]\modulo{\Delta^n_{j+\frac{1}{2}}}\\
    &+\frac{\lambda}{2}\sum_jF'(\tilde{\rho}^n_{j+\frac{1}{2}})\left(V^{n-h}_j-V^{n-h}_{j+1}\right)\modulo{\Delta^n_{j+\frac{1}{2}}}+\Delta t\mathcal{H}\sum_j\modulo{\Delta ^{n-h}_{j+\frac{1}{2}}}\\
    \leq& \ \sum_j\left[1+\frac{\lambda}{2}F'(\tilde{\rho}^n_{j+\frac{1}{2}})\left(V^{n-h}_{j-1}-V^{n-h}_{j+1}\right)\right]\modulo{\Delta^n_{j+\frac{1}{2}}}+\Delta t\mathcal{H}\sum_j\modulo{\Delta ^{n-h}_{j+\frac{1}{2}}}.
\end{align*}
Using the mean value theorem, the monotonicity of $\omega$ and Lemma~\ref{boundteo}, one can bound
\begin{eqnarray}
        \modulo{V^{n-h}_{j-1}-V^{n-h}_{j+1}}&\leq&\modulo{V^{n-h}_{j-1}-V^{n-h}_j}+\modulo{V^{n-h}_j-V^{n-h}_{j+1}}\nonumber\\
        &\leq&\norma{v'}\Delta x\left|\omega^0\rho^{n-h}_{j-1}+\sum_{k=1}^{+\infty}(\omega^k-\omega^{k-1})\rho^{n-h}_{j+k-1}\right|\nonumber\\
        &&+\norma{v'}\Delta x\left|\omega^0\rho^{n-h}_j+\sum_{k=1}^{+\infty}(\omega^k-\omega^{k-1})\rho^{n-h}_{j+k}\right|\label{stimamodulo_parziale}\\
        &\leq&4 \norma{v'}\norma{\omega}R\Delta x.\label{stimamodulo}
    \end{eqnarray}
     It follows that
\begin{equation}\label{inequality}
\sum_j\modulo{\Delta^{n+1}_{j+\frac{1}{2}}}\leq(1+\Delta t\mathcal{G})\sum_j\modulo{\Delta^n_{j+\frac{1}{2}}}+\Delta t\mathcal{H}\sum_j\modulo{\Delta^{n-h}_{j+\frac{1}{2}}},
\end{equation}
being $\mathcal{G}=2\norma{v'}\norma{\omega}R\left(1+R\norma{f'}\right)$. \\
Setting $\tv^n:=\sum_j\modulo{\Delta^{n}_{j+\frac{1}{2}}}$ and observing that $\tv^{-h}=\dots=\tv^{-1}=\tv^0$, from \eqref{inequality} we get
\begin{align*}
    \tv^1&\leq(1+\Delta t\mathcal{G})\tv^0+\Delta t\mathcal{H}\tv^0,\\
    \tv^2&\leq(1+\Delta t\mathcal{G})\tv^1+\Delta t\mathcal{H}\tv^0\\
    &\leq (1+\Delta t\mathcal{G})^2\tv^0 +\Delta t\mathcal{H}(1+(1+\Delta t\mathcal{G}))\tv^0 ,\\
    & ~~\vdots\\
    \tv^h& \leq (1+\Delta t\mathcal{G})^h\tv^0 +\Delta t\mathcal{H}\tv^0\sum_{k=0}^{h-1}(1+\Delta t\mathcal{G})^k\\
    &\leq (1+\Delta t\mathcal{G})^h\tv^0 +
    \left[(1+\Delta t\mathcal{M})^h -1\right] \tv^0,
\end{align*}
with
\begin{equation} \label{eq:M}
    \mathcal{M}:=\max\{\mathcal{H},\mathcal{G}\}
    = 2\,\norma{\omega}R\left(\norma{v'}+R\max\left\{
3\norma{v''}J_0,
\norma{v'}\norma{f'}
    \right\}\right).
\end{equation}
Recalling that $h=\tau/\Delta t$ and passing to the limit as $\Delta t\to 0$, we obtain
\begin{equation*}
    \tv\left(\rho^{\Delta x}(k\Delta t,\cdot)\right)\leq 
    \left( e^{\mathcal{G}\tau} + e^{\mathcal{M}\tau} -1 \right) \tv(\rho^0)
    \leq 
    \left( 2e^{\mathcal{M}\tau} -1 \right) \tv(\rho^0),
    \qquad k=0,\ldots,h.
\end{equation*}
Iterating the above argument for $k=(n-1)h+1,\ldots,nh$, we get
\begin{equation*}
    \tv\left(\rho^{\Delta x}(k\Delta t,\cdot)\right)\leq 
    \left( 2e^{\mathcal{M}\tau} -1 \right)^n \tv(\rho^0),
\end{equation*}
and, in general, setting $T= t + \lfloor T/\tau \rfloor \tau$
\begin{equation} \label{eq:TVestimate}
    \tv\left(\rho^{\Delta x}(T,\cdot)\right)\leq 
\left( 2e^{\mathcal{M}t} -1 \right)
    \left( 2e^{\mathcal{M}\tau} -1 \right)^{\lfloor T/\tau \rfloor} \tv(\rho^0).
\end{equation}  
Observe that
\[
\lim_{\tau\to 0}
\left( 2e^{\mathcal{M}(T-\lfloor T/\tau \rfloor \tau)} -1 \right)
    \left( 2e^{\mathcal{M}\tau} -1 \right)^{\lfloor T/\tau \rfloor}
    = e^{2 \mathcal{M} T},
\]
thus we recover the estimate for the non-delayed model \cite[Proposition 2]{ChiarelloGoatin2018}.
\end{proof}

\begin{remark}[Dependence on the parameters]\label{relazioneoss}
    The bound \eqref{eq:TVestimate} clearly indicates that the bigger the product $\mathcal{M}\tau$ is, the faster the total variation increases.
    We remark that the hypotheses on the convolution kernel imply that the value of $\norma{\omega}$ is linked to the look-ahead distance $L$. Indeed, if $\omega\in\C{1}([0,1];\R^+)$, then we can choose a re-scaled kernel function such as
    $$
    \omega_L(x)=\frac{1}{L}\omega\left(\frac{x}{L}\right).
    $$
    Thus, by~\eqref{eq:M}, we deduce that the positive constant $\mathcal{M}$ is dimensionally the inverse of time and it is such that $\mathcal{M}\sim 1/L$.
\end{remark}

Now, we assume the CFL condition
\begin{equation}\label{CFL HW}
\lambda\leq\frac{1}{V\left(1+R\norma{f'}\right)},
\end{equation}
and, using the same notation as above, we prove that similar $\BV$ estimates hold for the HW scheme.

\begin{proposition}[Spatial $\BV$-bound for the HW scheme]\label{spaceBVteo2}
    Let Assumption~\ref{hp} and the CFL condition \eqref{CFL HW} hold. Then, for any $\rho^0\in \BV(\R;[0,R])$, the numerical solution $\rho^{\Delta x}(t,\cdot)$ given by the HW scheme \eqref{schema2} has bounded total variation for $t\in[0,T]$, uniformly in $\Delta x$, for every time horizon $T>0$.
\end{proposition}

\begin{proof}
Our aim is to prove that the sequence $\tv^n=\sum_j\modulo{\Delta^n_{j+\frac{1}{2}}}$, given by \eqref{schema2}, satisfies the relation
\begin{equation*}
\tv^{n+1}\leq(1+\Delta t\mathcal{G})\tv^n+\Delta t\mathcal{H}\tv^{n-h}
\end{equation*} 
because, exactly as we did for the sequence \eqref{inequality} to obtain the bound \eqref{eq:TVestimate}, from this one can prove that for any $T=t+\lfloor T/\tau\rfloor\tau$ it holds
$$
    \tv\left(\rho^{\Delta x}(T,\cdot)\right)\leq 
\left( 2e^{\mathcal{M}t} -1 \right)
    \left( 2e^{\mathcal{M}\tau} -1 \right)^{\lfloor T/\tau \rfloor} \tv(\rho^0).
    $$
    Thus, from \eqref{schema2} and the mean value theorem we get
\begin{align}
\Delta^{n+1}_{j+\frac{1}{2}}=&\ \Delta^n_{j+\frac{1}{2}}-\lambda\left(\rho^n_{j+1}f(\rho^n_{j+2})V^{n-h}_{j+2}-2\rho^n_jf(\rho^n_{j+1})V^{n-h}_{j+1}+\rho^n_{j-1}f(\rho^n_j)V^{n-h}_j\right)\nonumber\\
=&\ \Delta ^n_{j+\frac{1}{2}}\nonumber\\
&-\lambda\left[\left(\rho^n_{j+1}f(\rho^n_{j+2})-\rho^n_jf(\rho^n_{j+1})\right)V^{n-h}_{j+2}-\left(\rho^n_jf(\rho^n_{j+1})-\rho^n_{j-1}f(\rho^n_j)\right)V^{n-h}_{j+1}\right]\nonumber\\
&-\lambda\left[\rho^n_jf(\rho^n_{j+1})\left(V^{n-h}_{j+2}-V^{n-h}_{j+1}\right)-\rho^n_{j-1}f(\rho^n_j)\left(V^{n-h}_{j+1}-V^{n-h}_j\right)\right]\nonumber\\
=&\ \Delta ^n_{j+\frac{1}{2}}\nonumber\\
&-\lambda\left[\left(\rho^n_{j+1}f(\rho^n_{j+2})-\rho^n_jf(\rho^n_{j+1})\right)V^{n-h}_{j+2}-\left(\rho^n_jf(\rho^n_{j+1})-\rho^n_{j-1}f(\rho^n_j)\right)V^{n-h}_{j+1}\right]\nonumber\\
&-\lambda\left(\rho^n_jf(\rho^n_{j+1})-\rho^n_{j-1}f(\rho^n_j)\right)\left(V^{n-h}_{j+2}-V^{n-h}_{j+1}\right)\nonumber\\
&-\lambda\rho^n_{j-1}f(\rho^n_j)\left[\left(V^{n-h}_{j+2}-V^{n-h}_{j+1}\right)-\left(V^{n-h}_{j+1}-V^{n-h}_j\right)\right]\nonumber\\
=&\ \Delta ^n_{j+\frac{1}{2}}\nonumber\\
&-\lambda\left(f(\rho^n_{j+2})\Delta^n_{j+\frac{1}{2}}+\rho^n_jf'(\tilde{\rho}^n_{j+\frac{3}{2}})\Delta^n_{j+\frac{3}{2}}\right)V^{n-h}_{j+2}\nonumber\\
&+\lambda\left(f(\rho^n_{j+1})\Delta^n_{j-\frac{1}{2}}+\rho^n_{j-1}f'(\tilde{\rho}^n_{j+\frac{1}{2}})\Delta^n_{j+\frac{1}{2}}\right)V^{n-h}_{j+1}\nonumber\\
&-\lambda\left(f(\rho^n_{j+1})\Delta^n_{j-\frac{1}{2}}+\rho^n_{j-1}f'(\tilde{\rho}^n_{j+\frac{1}{2}})\Delta^n_{j+\frac{1}{2}}\right)\left(V^{n-h}_{j+2}-V^{n-h}_{j+1}\right)\nonumber\\
&-\lambda\rho^n_{j-1}f(\rho^n_j)\left[\left(V^{n-h}_{j+2}-V^{n-h}_{j+1}\right)-\left(V^{n-h}_{j+1}-V^{n-h}_j\right)\right]\nonumber\\
=&\left[1-\lambda\left(f(\rho^n_{j+2})V^{n-h}_{j+2}-\rho^n_{j-1}f'(\tilde{\rho}^n_{j+\frac{1}{2}})V^{n-h}_{j+1}\right)\right]\Delta ^n_{j+\frac{1}{2}}\label{1}\\
&-\lambda\rho^n_jf'(\tilde{\rho}^n_{j+\frac{3}{2}})V^{n-h}_{j+2}\Delta^n_{j+\frac{3}{2}}+\lambda f(\rho^n_{j+1})V^{n-h}_{j+1}\Delta^n_{j-\frac{1}{2}}\label{2}\\
&-\lambda\left(f(\rho^n_{j+1})\Delta^n_{j-\frac{1}{2}}+\rho^n_{j-1}f'(\tilde{\rho}^n_{j+\frac{1}{2}})\Delta^n_{j+\frac{1}{2}}\right)\left(V^{n-h}_{j+2}-V^{n-h}_{j+1}\right)\label{3}\\
&-\lambda\rho^n_{j-1}f(\rho^n_j)\left[\left(V^{n-h}_{j+2}-V^{n-h}_{j+1}\right)-\left(V^{n-h}_{j+1}-V^{n-h}_j\right)\right]\label{4}.
\end{align}
Thanks to the CFL condition \eqref{CFL HW} and some cancellations, we have
\begin{equation*}
\sum_j\left(\modulo{\eqref{1}}+\modulo{\eqref{2}}\right)\leq\tv^n.
\end{equation*}
Moreover, using Remark \ref{HWremark}, as in \eqref{boundH} one can show that
\begin{equation*}
\sum_j\modulo{\eqref{4}}\leq\Delta t\mathcal{H}\tv^{n-h}.
\end{equation*}
To conclude the proof, recalling that as in \eqref{stimamodulo} we can prove for every $j\in\Z$ the bound 
\begin{align}
\modulo{V^{n-h}_{j+1}-V^{n-h}_j}& \leq \norma{v'}\Delta x\modulo{\sum_{k=1}^{+\infty}(\omega^{k-1}-\omega^k)\rho^{n-h}_{j+k}-\omega^0\rho^{n-h}_j}\label{stimamodulo_parziale2}\\
&\leq 2\norma{v'}\norma{\omega}R\Delta x\nonumber,
\end{align}
 we write
$$
\sum_j\modulo{\eqref{3}}\leq\lambda\sum_j\modulo{V^{n-h}_{j+3}-V^{n-h}_{j+2}}\modulo{\Delta^n_{j+\frac{1}{2}}}+\lambda R\norma{f'}\sum_j\modulo{V^{n-h}_{j+2}-V^{n-h}_{j+1}}\modulo{\Delta^n_{j+\frac{1}{2}}}\leq\Delta t\mathcal{G}\tv^n.
$$
\end{proof}

To prove an estimate for the discrete total variation in space and time, we need the following result.

\begin{lemma}[$\L1$ Lipschitz continuity in time]\label{L1contteo}
Let Assumption \ref{hp} and the CFL condition \eqref{CFL LF} hold. Then, for any $\rho^0\in \BV(\R;[0,R])$, the approximate solution constructed via the LF scheme~\eqref{schema} satisfies 
\begin{equation}\label{L1 LF}
\norma{\rho^{\Delta x}(T,\cdot)-\rho^{\Delta x}(T-t,\cdot)}_{1}\leq\mathcal{K}\,t
\end{equation}
for any $T>0$ and $t\in[0,T]$, with $\mathcal{K}$ given by \eqref{k LF}.
\end{lemma}

\begin{proof}
Let $N_T\in\mathbb{N}$ be such that $N_T\Delta t<T\leq(N_T+1)\Delta t$.
We recall from \eqref{schema} that for every $j\in\mathbb{Z}$ and $n=0,\dots,N_T-1$
    \begin{align}
    \rho^{n+1}_j-\rho^n_j=&\ \frac{\lambda\alpha}{2}\left(\rho^n_{j+1}-2\rho^n_j+\rho^n_{j-1}\right)\nonumber\\
    &-\frac{\lambda}{2}\left[F(\rho^n_{j+1})-F(\rho^n_{j-1})\right]V^{n-h}_{j+1}-\frac{\lambda}{2}F(\rho^n_{j-1})\left(V^{n-h}_{j+1}-V^{n-h}_{j-1}\right)\nonumber\\
    =&\ \frac{\lambda\alpha}{2}\left(\rho^n_{j+1}-\rho^n_j\right)-\frac{\lambda\alpha}{2}\left(\rho^n_j-\rho^n_{j-1}\right)\nonumber\\
    &-\frac{\lambda}{2}\left[F'(\tilde{\rho}^n_{j+\frac{1}{2}})\left(\rho^n_{j+1}-\rho^n_j\right)+F'(\tilde{\rho}^n_{j-\frac{1}{2}})\left(\rho^n_j-\rho^n_{j-1}\right)\right]V^{n-h}_{j+1}\nonumber\\
    &-\frac{\lambda}{2}F(\rho^n_{j-1})\left(V^{n-h}_{j+1}-V^{n-h}_{j-1}\right)\nonumber\\
    =&\ \frac{\lambda}{2}\left[\alpha-F'(\tilde{\rho}^n_{j+\frac{1}{2}})V^{n-h}_{j+1}\right]\left(\rho^n_{j+1}-\rho^n_j\right)-\frac{\lambda}{2}\left[\alpha+F'(\tilde{\rho}^n_{j-\frac{1}{2}})V^{n-h}_{j+1}\right]\left(\rho^n_j-\rho^n_{j-1}\right)\nonumber\\
    &+\frac{\lambda}{2}F(\rho^n_{j-1})\left(V^{n-h}_{j-1}-V^{n-h}_{j+1}\right),\label{delta rho no}
    \end{align}
with again $\tilde{\rho}^n_{j+\frac{1}{2}}$  between $\rho^n_j$ and $\rho^n_{j+1}$, and $\tilde{\rho}^n_{j-\frac{1}{2}}$  between $\rho^n_{j-1}$ and $\rho^n_j$.
 Thus, from \eqref{stimamodulo_parziale}, we get
 \begin{align}
 &\modulo{\rho^{n+1}_j-\rho^n_j} \nonumber\\
 \leq&\ \frac{\lambda}{2}
 \left[\alpha+\left(1+R\norma{f'}\right)V\right]
 \left(\modulo{\rho^n_{j+1}-\rho^n_j}+
 \modulo{\rho^n_j-\rho^n_{j-1}}\right)\nonumber\\
 &+\frac{\lambda}{2}R \norma{v'}\Delta x\left[\omega^0\left(\modulo{\rho^{n-h}_{j-1}}+\modulo{\rho^{n-h}_j}\right)+\sum^{+\infty}_{k=1}(\omega^{k-1}-\omega^k)\left(\modulo{\rho_{j+k-1}^{n-h}}+\modulo{\rho_{j+k}^{n-h}}\right)\right].\label{delta rho}
 \end{align}
Now, we fix $t=m\Delta t$, with $m\leq N_T$.
Using the last inequality,  Lemma \ref{L1boundteo} and Proposition \ref{spaceBVteo}, we obtain 
\begin{align}
\sum_j\Delta x\modulo{\rho^{N_T}_j-\rho^{N_T-m}_j}\leq&\sum_{n=N_T-m}^{N_T-1}\sum_j\Delta x\modulo{\rho^{n+1}_j-\rho^n_j}\nonumber\\
\leq& \  \left[\alpha+\left(1+R\norma{f'}\right)V\right]
\sum_{n=N_T-m}^{N_T-1}\sum_j\Delta t\modulo{\rho^n_{j+1}-\rho^n_j}\nonumber\\
&+R \norma{v'}\omega^0\sum_{n=N_T-m}^{N_T-1}\sum_j\Delta t\Delta x\modulo{\rho^{n-h}_j}\nonumber\\
&+R \norma{v'}\sum^N_{k=1}(\omega^{k-1}-\omega^k)\sum_{n=N_T-m}^{N_T-1}\sum_j\Delta t\Delta x\modulo{\rho_{j+k}^{n-h}}\nonumber\\
\leq& \  \left[\alpha+\left(1+R\norma{f'}\right)V\right]
\sum_{n=N_T-m}^{N_T-1}\Delta t\sum_j\modulo{\rho^n_{j+1}-\rho^n_j}\nonumber\\
&+2R \norma{v'} \omega^0\sum_{n=N_T-m}^{N_T-1}\Delta t \sum_j\Delta x\modulo{\rho^{n-h}_j},\nonumber\\
\leq&\ t\left[\alpha+\left(1+R\norma{f'}\right)V\right]
\sup_{s\in[0,T]}\tv(\rho^{\Delta x}(s,\cdot))\nonumber\\
&+2tR \norma{v'}\norma{\omega}\sup_{s\in[0,T]}\norma{\rho^{\Delta x}(s,\cdot)}_1\leq\mathcal{K}t,\label{BVstima}
\end{align}
with
\begin{equation}\label{k LF}
\mathcal{K}:=\left[\alpha+\left(1+R\norma{f'}\right)V\right] C(T,\norma{\omega},\tau)\tv(\rho^0)+2R\norma{\omega}\norma{v'} \norma{\rho^0}_1,
\end{equation}
where the positive constant 
\begin{align}
C(T,\norma{\omega},\tau):=&\sup_{s\in[0,T]}\left( 2e^{\mathcal{M}(s-\lfloor s/\tau \rfloor\tau)} -1 \right)\left( 2e^{\mathcal{M}\tau} -1 \right)^{\lfloor s/\tau \rfloor}\nonumber\\
=&\left( 2e^{\mathcal{M}(T-\lfloor T/\tau \rfloor\tau)} -1 \right)\left( 2e^{\mathcal{M}\tau} -1 \right)^{\lfloor T/\tau \rfloor}\label{c}
\end{align}
is given by \eqref{eq:TVestimate}.
\end{proof}

\begin{remark} 
The $\L 1$ Lipschitz continuity in time holds also for the HW scheme under Assumption \ref{hp} and the CFL condition \eqref{CFL HW}. This follows from the fact that using \eqref{schema2} we get
\begin{align*}
\rho^{n+1}_j-\rho^n_j=&-\lambda\left(\rho^n_jf(\rho^n_{j+1})V^{n-h}_{j+1}-\rho^n_{j-1}f(\rho^n_j)V^{n-h}_j\right)\\
=&-\lambda\left[\rho^n_jf(\rho^n_{j+1})-\rho^n_{j-1}f(\rho^n_j)\right]V^{n-h}_{j+1}-\lambda\rho^n_{j-1}f(\rho^n_j)\left(V^{n-h}_{j+1}-V^{n-h}_j\right)\\
=&-\lambda f(\rho^n_{j+1})V^{n-h}_{j+1}\left(\rho^n_j-\rho^n_{j-1}\right)-\lambda\rho^n_{j-1}f'(\tilde{\rho}^n_{j+\frac{1}{2}})V^{n-h}_{j+1}\left(\rho^n_{j+1}-\rho^n_j\right)\\
&-\lambda\rho^n_{j-1}f(\rho^n_j)\left(V^{n-h}_{j+1}-V^{n-h}_j\right),
\end{align*}
with $\tilde{\rho}^n_{j+\frac{1}{2}}$ defined as before.
Thus, similarly to \eqref{BVstima}, we obtain
$$
\sum_j\Delta x\modulo{\rho^{N_T}_j-\rho^{N_T-m}_j}\leq\mathcal{K}t,
$$
with
\begin{equation}\label{k HW}
\mathcal{K}:=V \left(1+R\norma{f'}\right) C(T,\norma{\omega},\tau)\tv(\rho^0)+2R\norma{\omega}\norma{v'} \norma{\rho^0}_1,
\end{equation}
and $C(T,\norma{\omega},\tau)$ defined as in \eqref{c}.
\end{remark}

We can now provide an estimate for the discrete total variation in space and time. We give the following proof using the LF scheme, but the result can be identically proven for the HW scheme under the CFL condition \eqref{CFL HW} and for $\mathcal{K}$ given by \eqref{k HW}.

\begin{proposition}[$\BV$ estimates in space and time]\label{BVteo}
Let Assumption \ref{hp} and the CFL condition \eqref{CFL LF} hold. Then, for any $\rho^0\in \BV(\R;[0,R])$, the numerical solution $\rho^{\Delta x}$ has bounded total variation on $[0,T]\times\R$, uniformly in $\Delta x$ and $\Delta t$, for any time horizon $T>0$.
\end{proposition}

\begin{proof}
    If $T\leq\Delta t$, then TV$(\rho^{\Delta x})\leq T\cdot$TV$(\rho^0)$. For $T$ such that $N_T\Delta t<T\leq(N_T+1)\Delta t$ with $N_T\in\mathbb{N}$, we have
    \begin{eqnarray*}
        \tv(\rho^{\Delta x})&=&\sum_{n=0}^{N_T-1}\sum_{j\in\mathbb{Z}}\Delta t\modulo{\rho^n_{j+1}-\rho^n_j}+(T-N_T\Delta t)\sum_{j\in\mathbb{Z}}\modulo{\rho^{N_T}_{j+1}-\rho_j^{N_T}}\\
        &&+\sum_{n=0}^{N_T-1}\sum_{j\in\mathbb{Z}}\Delta x\modulo{\rho^{n+1}_j-\rho^n_j}.
    \end{eqnarray*}
    The first two terms together are bounded by $T\sup_{t\in[0,T]}\tv(\rho^{\Delta x}(t,\cdot))$. The result follows from Lemma \ref{L1contteo} since, as in \eqref{BVstima}, one can bound the last term as
$$
\sum_{n=0}^{N_T-1}\sum_j\Delta x\modulo{\rho^{n+1}_j-\rho^n_j}\leq\mathcal{K}T.
$$
\end{proof}

\section{Well-posedness of entropy weak solutions}
\label{sec:convergence}

In order to prove the existence of an entropy solution, following \cite{ChiarelloGoatin2018} we derive a discrete entropy inequality for the approximate solution generated by the LF scheme \eqref{schema}, which is used to prove that the limit of the LF approximations is indeed an entropy solution in the sense of Definition \ref{entropy}. Let us denote
\begin{eqnarray*}
    G_{j+\frac{1}{2}}(u,w)&=&\frac{1}{2}F(u)V^{n-h}_j+\frac{1}{2}F(w)V^{n-h}_{j+1}-\frac{\alpha}{2}(w-u),\\
    F^\kappa_{j+\frac{1}{2}}(u,w)&=&G_{j+\frac{1}{2}}(u\wedge\kappa,w\wedge\kappa)-G_{j+\frac{1}{2}}(u\vee\kappa,w\vee\kappa), \\
    &=& \frac{1}{2} \sgn(u-\kappa) \left(F(u)-F(\kappa)\right) V_j^{n-h} \\
    && + \frac{1}{2} \sgn(w-\kappa) \left(F(w)-F(\kappa)\right) V_{j+1}^{n-h} \\
    && +\frac{\alpha}{2}\left(\modulo{w-\kappa}-\modulo{u-\kappa}\right)
\end{eqnarray*}
with $a\wedge b=\max\{a,b\}$ and $a\vee b=\min\{a,b\}$.

\begin{proposition}[Discrete entropy inequality]
Given Assumption \ref{hp}, let $\rho^n_j$, $j\in\mathbb{Z}$, $n\in\{-h,\dots,0\}\cup\mathbb{N}$, be given by \eqref{schema}. Then, if the CFL condition \eqref{CFL LF} is satisfied, we have
\begin{align}
    \modulo{\rho^{n+1}_j-\kappa}-\modulo{\rho^n_j-\kappa}&+\lambda\left(F^\kappa_{j+\frac{1}{2}}(\rho^n_j,\rho^n_{j+1})-F^\kappa_{j-\frac{1}{2}}(\rho^n_{j-1},\rho^n_j)\right)\nonumber\\
    &+\frac{\lambda}{2}\sgn(\rho^{n+1}_j-\kappa)F(\kappa)\left(V^{n-h}_{j+1}-V^{n-h}_{j-1}\right)\leq 0,\label{disentropy}
\end{align}
for all $j\in\mathbb{Z}$, $n\in\mathbb{N}_0$, and $\kappa\in\R$.
\end{proposition}

\begin{proof}
The proof follows closely \cite[Proposition 3.4]{ChiarelloGoatin2018}. We detail it below for sake of completeness.
    We set
    $$
    \tilde{H}_j(u,w,z)=w-\lambda\big(G_{j+\frac{1}{2}}(w,z)-G_{j-\frac{1}{2}}(u,w)\big),
    $$
    that is monotone non-decreasing in all its variables, since the derivatives are
    \begin{align*}
    \frac{\partial\tilde{H}_j}{\partial u}&=\frac{\lambda}{2}\left(\alpha+F'(u)V^{n-h}_{j-1}\right),\\
    \frac{\partial\tilde{H}_j}{\partial w}&=1-\lambda\alpha,\\
    \frac{\partial\tilde{H}_j}{\partial z}&=\frac{\lambda}{2}\left(\alpha-F'(z)V^{n-h}_{j+1}\right),
    \end{align*}
    and they are all non-negative because of the CFL condition.
    We have the identity
    \begin{align*}
        \tilde{H}_j(\rho^n_{j-1}\wedge\kappa,\rho^n_j\wedge\kappa,\rho^n_{j+1}\wedge\kappa)&-\tilde{H}_j(\rho^n_{j-1}\vee\kappa,\rho^n_j\vee\kappa,\rho^n_{j+1}\vee\kappa)\\
        &=\modulo{\rho^n_j-\kappa}-\lambda\big(F^\kappa_{j+\frac{1}{2}}(\rho^n_j,\rho^n_{j+1})-F^\kappa_{j-\frac{1}{2}}(\rho^n_{j-1},\rho^n_j)\big).
    \end{align*}
    By monotonicity and \eqref{schema} we get
    \begin{align*}
        \tilde{H}_j&(\rho^n_{j-1}\wedge\kappa,\rho^n_j\wedge\kappa,\rho^n_{j+1}\wedge\kappa)-\tilde{H}_j(\rho^n_{j-1}\vee\kappa,\rho^n_j\vee\kappa,\rho^n_{j+1}\vee\kappa)\\
        &=\tilde{H}_j(\rho^n_{j-1},\rho^n_j,\rho^n_{j+1})\wedge\tilde{H}_j(\kappa,\kappa,\kappa)-\tilde{H}_j(\rho^n_{j-1},\rho^n_j,\rho^n_{j+1})\vee\tilde{H}_j(\kappa,\kappa,\kappa)\\
        &=\modulo{\tilde{H}_j(\rho^n_{j-1},\rho^n_j,\rho^n_{j+1})-\tilde{H}_j(\kappa,\kappa,\kappa)}\\
        &=\sgn\left(\tilde{H}_j(\rho^n_{j-1},\rho^n_j,\rho^n_{j+1})-\tilde{H}_j(\kappa,\kappa,\kappa)\right)\cdot\left(\tilde{H}_j(\rho^n_{j-1},\rho^n_j,\rho^n_{j+1})-\tilde{H}_j(\kappa,\kappa,\kappa)\right)\\
        &=\sgn\left(\rho^{n+1}_j-\kappa+\frac{\lambda}{2}F(\kappa)\left(V^{n-h}_{j+1}-V^{n-h}_{j-1}\right)\right)\cdot\left(\rho^{n+1}_j-\kappa+\frac{\lambda}{2}F(\kappa)\left(V^{n-h}_{j+1}-V^{n-h}_{j-1}\right)\right)\\
        &\geq\sgn(\rho^{n+1}_j-\kappa)\cdot\left(\rho^{n+1}_j-\kappa+\frac{\lambda}{2}F(\kappa)\left(V^{n-h}_{j+1}-V^{n-h}_{j-1}\right)\right)\\
        &=\modulo{\rho^{n+1}_j-\kappa}+\frac{\lambda}{2}\sgn(\rho^{n+1}_j-\kappa)F(\kappa)\left(V^{n-h}_{j+1}-V^{n-h}_{j-1}\right),
    \end{align*}
    which gives \eqref{disentropy}.
\end{proof}

Given the entropy inequality and the $\BV$ estimates, we are able to propose a proof of existence of solutions, following the steps of \cite[Theorem 1.2]{ChiarelloGoatin2019}. 
The following result holds for every time horizon $T>0$. 

\begin{theorem}[Existence]\label{E1}
 Given Assumption \ref{hp}, for any $\rho^0\in \BV(\R;[0,R])$ and $T>0$ the model \eqref{delay_bis}-\eqref{eq:initial datum} admits an entropy weak solution $\rho$ in the sense of Definition \ref{entropy}, such that
 \begin{align}
   &  \rho(t,x) \in [0,R] &\hbox{for a.e. } x\in\R, t\in [0,T], \label{eq:linfty} \\
   &\norma{\rho(t,\cdot)}_1 = \|\rho^0\|_1 &\hbox{for } t\in [0,T], \label{eq:l1} \\
   & \tv (\rho(t,\cdot)) \leq \left( 2e^{\mathcal{M}(t-\lfloor t/\tau \rfloor \tau)} -1 \right)
    \left( 2e^{\mathcal{M}\tau} -1 \right)^{\lfloor t/\tau \rfloor} \tv(\rho^0)
    &\hbox{for } t\in [0,T], \label{eq:TV} \\
    &\norma{\rho(t_1,\cdot)-\rho(t_2,\cdot)}_1
    \leq \mathcal{K} \modulo{t_1 -t_2} &\hbox{for } t_1,t_2\in\, [0,T], \label{eq:L1time} 
 \end{align}
 for constants $\mathcal{M}$ and $\mathcal{K}$ given in~\eqref{eq:M} and~\eqref{k LF}, respectively.
\end{theorem}

\begin{proof}
By Lemma \ref{boundteo} and Remark~\ref{HWremark}, we know that the the approximate solutions $\rho^{\Delta x}$ constructed via the LF or HW schemes are uniformly bounded on $[0,T]\times\R$. Moreover, Proposition~\ref{BVteo} guarantees that the numerical solutions also have uniformly bounded total variation. Thus, from Helly's Theorem we get that there exists a subsequence of the numerical approximations, still denoted by $\rho^{\Delta x}$, that converges in the $\Lloc{1}$-norm to some $\rho\in\BV([0,T]\times\R;[0,R])$ as $\Delta x \searrow 0$.
In the following, we apply the classical procedure of Lax-Wendroff theorem to prove that the limit function $\rho$ is an entropy weak solution of~\eqref{delay_bis}-\eqref{eq:initial datum} in the sense of Definition~\ref{entropy}. 
Let $\phi\in\Cc 1([0,T[\,\times\R;\R^+)$ be a test function, and let us assume that the grid for the approximation is such that $N_T\Delta t<T\leq(N_T+1)\Delta t$. By multiplying \eqref{disentropy} by $\phi^n_j:=\phi(t_n,x_j)$ and summing on $n=0,\dots N_T$ and $j\in\mathbb{Z}$, we get
\begin{align*}
\sum_{n=0}^{N_T-1}\sum_j\phi^n_j\left(\modulo{\rho^{n+1}_j-\kappa}-\modulo{\rho^n_j-\kappa}\right)&+\lambda\sum_{n=0}^{N_T-1}\sum_j\phi^n_j\left(F^\kappa_{j+\frac{1}{2}}(\rho^n_j,\rho^n_{j+1})-F^\kappa_{j-\frac{1}{2}}(\rho^n_{j-1},\rho^n_j)\right)\\
    &+\lambda\sum_{n=0}^{N_T-1}\sum_j\phi^n_j\sgn(\rho^{n+1}_j-\kappa)F(\kappa)\frac{V^{n-h}_{j+1}-V^{n-h}_{j-1}}{2}\leq 0.
\end{align*}
Then, summing by parts, we obtain
\begin{align*}
0\leq&\sum_j\phi^0_j\modulo{\rho^0_j-\kappa}+\sum_{n=1}^{N_T-1}\sum_j(\phi^n_j-\phi^{n-1}_j)\modulo{\rho^n_j-\kappa}\\
& +\lambda\sum_{n=0}^{N_T-1}\sum_j(\phi^n_{j+1}-\phi^n_j)\,F^\kappa_{j+\frac{1}{2}}(\rho^n_j,\rho^n_{j+1})-\lambda\sum_{n=0}^{N_T-1}\sum_j\sgn(\rho^{n+1}_j-\kappa)F(\kappa)\frac{V^{n-h}_{j+1}-V^{n-h}_{j-1}}{2}\,\phi^n_j
\end{align*}
and multiplying by $\Delta x$
\begin{align}
0\leq&\ \Delta x\sum_j\phi^0_j\modulo{\rho^0_j-\kappa}+\Delta x\Delta t\sum_{n=1}^{N_T-1}\sum_j\frac{\phi^n_j-\phi^{n-1}_j}{\Delta t}\modulo{\rho^n_j-\kappa}\label{part1}\\
&+\Delta x\Delta t\sum_{n=0}^{N_T-1}\sum_j\frac{\phi^n_{j+1}-\phi^n_j}{\Delta x}\left[F^\kappa_{j+\frac{1}{2}}(\rho^n_j,\rho^n_{j+1})-\sgn(\rho^n_j-\kappa)\big(F(\rho^n_j)-F(\kappa)\big)V^{n-h}_j\right]\label{part2}\\
&+\Delta x\Delta t\sum_{n=0}^{N_T-1}\sum_j\frac{\phi^n_{j+1}-\phi^n_j}{\Delta x}\sgn(\rho^n_j-\kappa)\big(F(\rho^n_j)-F(\kappa)\big)V^{n-h}_j\label{part3}\\
&-\Delta x\Delta t\sum_{n=0}^{N_T-1}\sum_j\sgn(\rho^n_j-\kappa)F(\kappa)\frac{V^{n-h}_{j+1}-V^{n-h}_{j-1}}{2\Delta x}\,\phi^n_j\label{part4}\\
&-\frac{1}{2}\Delta tF(\kappa)\sum_{n=0}^{N_T-1}\sum_j\left[\sgn(\rho^{n+1}_j-\kappa)-\sgn(\rho^{n}_j-\kappa)\right]\left(V^{n-h}_{j+1}-V^{n-h}_{j-1}\right)\, \phi^n_j.\label{part5}
\end{align}
Clearly, we have
$$
\eqref{part1}\rightarrow\int_\R\modulo{\rho^0(x)-\kappa}\phi(0,x)\d x+\int_0^T\int_\R\modulo{\rho-\kappa} \del_t\phi \d x\d t,
$$
$$
\eqref{part3}\rightarrow\int_0^T\int_\R\sgn(\rho-\kappa)\big(F(\rho)-F(\kappa)\big)v((\rho\ast\omega)(t-\tau,x))\del_x\phi \d x\d t,
$$
and
$$
\eqref{part4}\rightarrow-\int_0^T\int_\R\sgn(\rho-\kappa)F(\kappa)\del_xv((\rho\ast\omega)(t-\tau,x))\phi\d x\d t,
$$
as $\Delta x\rightarrow 0$. 
Next, we need to prove that both \eqref{part2} and \eqref{part5} converge to zero.
Let us first focus on~$\eqref{part2}$.
Since from \eqref{stimamodulo} it follows that $V^{n-h}_{j+1}-V^{n-h}_j=\mathcal{O}(\Delta x)$, and since the mean value theorem guarantees the bound
\begin{align*}
    \frac{1}{2}\modulo{\sgn(\rho^n_{j+1}-\kappa) \left(F(\rho^n_{j+1})-F(\kappa)\right)}\leq&\ \frac{1}{2}\norma{F'}\modulo{\rho^n_{j+1}-\kappa}\leq\frac{1}{2}\left(1+\left(R\wedge\modulo{\kappa}\right)\norma{f'}\right)(R+|\kappa|),
\end{align*}
then we get
\begin{align*}
    F^\kappa_{j+\frac{1}{2}}(\rho^n_j,&\rho^n_{j+1})-\sgn(\rho^n_j-\kappa) \left(F(\rho^n_j)-F(\kappa)\right) V_j^{n-h}=\\
    =&\ \frac{1}{2} \sgn(\rho^n_j-\kappa) \left(F(\rho^n_j)-F(\kappa)\right) V_j^{n-h}+ \frac{1}{2} \sgn(\rho^n_{j+1}-\kappa) \left(F(\rho^n_{j+1})-F(\kappa)\right) V_j^{n-h} \\
    &+\  \frac{1}{2} \sgn(\rho^n_{j+1}-\kappa) \left(F(\rho^n_{j+1})-F(\kappa)\right) \left(V_{j+1}^{n-h}-V_j^{n-h}\right)\\
    &-\ \sgn(\rho^n_j-\kappa) \left(F(\rho^n_j)-F(\kappa)\right) V_j^{n-h}\\
    & +\frac{\alpha}{2}\left(\modulo{\rho^n_{j+1}-\kappa}-\modulo{\rho^n_j-\kappa}\right)\\
    =&\ \frac{1}{2}\sgn(\rho^n_{j+1}-\kappa)\left(F(\rho^n_{j+1})-F(\kappa) \right)V^{n-h}_j-\frac{1}{2}\sgn(\rho^n_j-\kappa)\left(F(\rho^n_j)-F(\kappa)\right)V^{n-h}_j\\
    &+\frac{\alpha}{2}\left(\modulo{\rho^n_{j+1}-\kappa}-\modulo{\rho^n_j-\kappa}\right)+\mathcal{O}(\Delta x),
\end{align*}
and this implies
\begin{align}
\Big|F^\kappa_{j+\frac{1}{2}}(\rho^n_j,\rho^n_{j+1})&-\sgn(\rho^n_j-\kappa) \left(F(\rho^n_j)-F(\kappa)\right) V_j^{n-h}\Big|\nonumber\\
\leq&\ \frac{1}{2} V\modulo{\sgn(\rho^n_{j+1}-\kappa) \left(F(\rho^n_{j+1})-F(\kappa)\right)- \sgn(\rho^n_j-\kappa) \left(F(\rho^n_j)-F(\kappa)\right)} \label{Q} \\
&+\frac{\alpha}{2}\modulo{\rho^n_{j+1}-\rho^n_j}+\mathcal{O}(\Delta x)\nonumber\\
    \leq&\  \frac{1}{2}\left[\alpha+\left(1+\left(R\wedge\modulo{\kappa}\right) \norma{f'}\right)V\right]\modulo{\rho^n_{j+1}-\rho^n_j}+\mathcal{O}(\Delta x),\nonumber
\end{align}
where we used the following remarks:
\begin{itemize}

    \item if $\sgn(\rho^n_{j+1}-\kappa)=\sgn(\rho^n_{j+1}-\kappa)$, then
    \begin{align*}
        \eqref{Q}=&\ \frac{1}{2} V \modulo{F(\rho^n_{j+1})-F(\rho^n_j)}
        \leq\frac{1}{2}V\left(1+\left(R\wedge\modulo{\kappa}\right)\norma{f'}\right)\modulo{\rho^n_{j+1}-\rho^n_j};
    \end{align*}

    \item if $\sgn(\rho^n_{j+1}-\kappa)=-\sgn(\rho^n_j-\kappa)$, then
    \begin{align*}
        \eqref{Q}& \leq\frac{1}{2} V\modulo{F(\rho^n_{j+1})-F(\kappa)+F(\rho^n_j)-F(\kappa)}\\
        &\leq\frac{1}{2}V\norma{F'}
        \left(\modulo{\rho^n_{j+1}-\kappa}+\modulo{\rho^n_j-\kappa}\right)\\
        &\leq \frac{1}{2} V \left(1+\left(R\wedge\modulo{\kappa}\right)\norma{f'}\right)\modulo{\rho^n_{j+1}-\rho^n_j}.
    \end{align*}
\end{itemize}
We set $X>0$ such that $\phi(t,x)=0$ for $|x|>X$ and a couple of indexes $j_0,j_1\in\mathbb{Z}$ such that $\phi^n_j=0$ if $j$ is not in $[j_0,j_1]$.
Thus, we conclude
\begin{align*}
    |\eqref{part2}|\leq&\ \Delta x\Delta t\norma{\partial_x\phi}\sum_{n=0}^{N_T-1}\sum_{j=j_0}^{j_1}\modulo{F^\kappa_{j+\frac{1}{2}}(\rho^n_j,\rho^n_{j+1})-\sgn(\rho^n_j-\kappa) \left(F(\rho^n_j)-F(\kappa)\right) V_j^{n-h}}\\
    \leq&\ \frac{1}{2}\Delta x\Delta t\norma{\partial_x\phi}\left[\alpha+\left(1+(R\wedge\modulo{\kappa})\norma{f'}\right)V\right]\sum_{n=0}^{N_T-1}\sum_{j=j_0}^{j_1}\modulo{\rho^n_{j+1}-\rho^n_j}+\mathcal{O}(\Delta x)\\
    =&\ \mathcal{O}(\Delta x),
\end{align*}
which follows from the bound
\begin{align*}
    \Delta t\sum_{n=0}^{N_T-1}\sum_{j=j_0}^{j_1}\modulo{\rho^n_{j+1}-\rho^n_j}\leq T\sup_{t\in [0,T]} \tv(\rho^{\Delta x}(t,\cdot))\leq TC(T,\norma{\omega},\tau)\tv(\rho^0),
\end{align*}
for $C(T,\norma{\omega},\tau)$ defined as in \eqref{c}.
Finally, we focus on \eqref{part5}, following \cite{Betancourt2011}. Summing by parts again yields
\begin{align}
\eqref{part5}=&\ \frac{1}{2}\Delta tF(\kappa)\sum_{n=1}^{N_T-1}\sum_j\sgn(\rho^n_j-\kappa)\left[\left(V^{n-h}_{j+1}-V^{n-h}_{j-1}\right)\phi^n_j-\left(V^{n-h-1}_{j+1}-V^{n-h-1}_{j-1}\right)\phi^{n-1}_j\right]\nonumber\\
&+\frac{1}{2}\Delta tF(\kappa)\sum_j\sgn(\rho^0_j-\kappa)\left(V^{-h}_{j+1}-V^{-h}_{j-1}\right)\phi^0_j\nonumber\\
=&\ \frac{1}{2}\Delta tF(\kappa)\sum_{n=1}^{N_T-1}\sum_j\sgn(\rho^n_j-\kappa)\left[\left(V^{n-h}_{j+1}-V^{n-h}_{j-1}\right)-\left(V^{n-h-1}_{j+1}-V^{n-h-1}_{j-1}\right)\right]\phi^{n-1}_j\nonumber\\
&+\frac{1}{2}\Delta t\Delta xF(\kappa)\sum_{n=1}^{N_T-1}\sum_j\sgn(\rho^n_j-\kappa)\left(V^{n-h}_{j+1}-V^{n-h}_{j-1}\right)\frac{\phi^n_j-\phi^{n-1}_j}{\Delta x}\nonumber\\
&+\frac{1}{2}\Delta tF(\kappa)\sum_j\sgn(\rho^0_j-\kappa)\left(V^0_{j+1}-V^0_{j-1}\right)\phi^0_j.\label{perapp}
\end{align}
 Using again that $V^n_{j+1}-V^n_j=\mathcal{O}(\Delta x)$ for all $n\geq-h$, then we get
\begin{align}
    \eqref{part5}=&\ \frac{1}{2}\Delta tF(\kappa)\sum_{n=1}^{N_T-1}\sum_j\sgn(\rho^n_j-\kappa)\left[\left(V^{n-h}_{j+1}-V^{n-h}_{j-1}\right)-\left(V^{n-h-1}_{j+1}-V^{n-h-1}_{j-1}\right)\right]\phi^{n-1}_j\nonumber\\
    &+\mathcal{O}(\Delta x+\Delta t).\label{stimapart5}
\end{align}
Moreover, similarly to what we did in the proof of Proposition \ref{spaceBVteo} for the term $(\ast)$, we have
\begin{align}
    \big(V^{n-h}_{j+1}-&V^{n-h}_{j-1}\big)-\left(V^{n-h-1}_{j+1}-V^{n-h-1}_{j-1}\right)\nonumber\\
    =&\ v'(\xi^{n-h}_j)\Delta x\sum_{k=0}^{+\infty}\omega^k\left(\rho^{n-h}_{j+k+1}-\rho^{n-h}_{j+k-1}\right)
-v'(\xi^{n-h-1}_j)\Delta x\sum_{k=0}^{+\infty}\omega^k\left(\rho^{n-h-1}_{j+k+1}-\rho^{n-h-1}_{j+k-1}\right)\nonumber\\
=&\ \Delta x~\left[v'(\xi^{n-h}_j)-v'(\xi^{n-h-1}_j)\right]\sum_{k=0}^{+\infty}\omega^k\left(\rho^{n-h}_{j+k+1}-\rho^{n-h}_{j+k-1}\right)\nonumber\\
&+\Delta x~v'(\xi^{n-h-1}_j)\sum_{k=0}^{+\infty}\omega^k\left[\left(\rho^{n-h}_{j+k+1}-\rho^{n-h}_{j+k-1}\right)-\left(\rho^{n-h-1}_{j+k+1}-\rho^{n-h-1}_{j+k-1}\right)\right]\nonumber\\
=&\ \Delta x~v''(\bar{\xi}_j)\left[\xi^{n-h}_j-\xi^{n-h-1}_j\right]\sum_{k=0}^{+\infty}\omega^k\left(\rho^{n-h}_{j+k+1}-\rho^{n-h}_{j+k-1}\right)\nonumber\\
&+\Delta x~v'(\xi^{n-h-1}_j) \Big[\sum_{k=1}^{N}(\omega^{k-1}-\omega^{k+1})\left(\rho^{n-h}_{j+k}-\rho^{n-h-1}_{j+k}\right)\nonumber\\
&\qquad\qquad\qquad\qquad-\omega^0\left(\rho^{n-h}_{j-1}-\rho^{n-h-1}_{j-1}\right)-\omega^1\left(\rho^{n-h}_j-\rho^{n-h-1}_j\right)\Big]\label{stimavelocita},
\end{align}
being $\xi^{n-h}_j$ between $\Delta x\sum_{k=0}^{+\infty}\omega^k\rho^{n-h}_{j+k+1}$ and $\Delta x\sum_{k=0}^{+\infty}\omega^k\rho^{n-h}_{j+k-1}$, and $\xi^{n-h-1}_j$ between $\Delta x\sum_{k=0}^{+\infty}\omega^k\rho^{n-h-1}_{j+k+1}$ and $\Delta x\sum_{k=0}^{+\infty}\omega^k\rho^{n-h-1}_{j+k-1}$, and for $\bar{\xi}_j$ between $\xi^{n-h}_j$ and $\xi^{n-h-1}_j$.
Regarding the first term, we remark that 
\begin{equation}
    \sum_{k=0}^{+\infty}\omega^k\modulo{\rho^{n-h}_{j+k+1}-\rho^{n-h}_{j+k-1}}
    \leq\norma{\omega}\sup_{t\in[0,T]}\tv(\rho^{\Delta x}(t,\cdot))\leq\norma{\omega}C(T,\norma{\omega},\tau)\tv(\rho^0),\label{stimavelocitaaux}
\end{equation} 
and we underline that $\xi^{n-h}_j-\xi^{n-h-1}_j=\mathcal{O}(\Delta x+\Delta t)$ (details in Appendix \ref{sec:app}). Moreover, using the same notation as above, \eqref{delta rho no} implies
\begin{equation*}
\rho^{n-h}_j-\rho^{n-h-1}_j= \lambda\mathcal{R}^n_j\left(\rho^{n-h-1}_{j+1}-\rho^{n-h-1}_j\right)-\lambda\mathcal{L}^n_j\left(\rho^{n-h-1}_j-\rho^{n-h-1}_{j-1}\right)+\mathcal{O}(\Delta t),
\end{equation*}
for every $j\in\mathbb{Z}$, being $\mathcal{R}^n_j=\frac{1}{2}[\alpha-F'(\tilde{\rho}^{n-h-1}_{j+\frac{1}{2}})V^{n-2h-1}_{j+1}]$ and $\mathcal{L}^n_j=\frac{1}{2}[\alpha-F'(\tilde{\rho}^{n-h-1}_{j-\frac{1}{2}})V^{n-2h-1}_{j+1}]$.
Thus, we have
\begin{align*}
    \big(V^{n-h}_{j+1}&-V^{n-h}_{j-1}\big)-\left(V^{n-h-1}_{j+1}-V^{n-h-1}_{j-1}\right)=\mathcal{O}(\Delta x^2+\Delta x\Delta t)\\
    &+\lambda\Delta x~v'(\xi^{n-h-1}_j)\Big[\sum_{k=1}^N\mathcal{R}^n_{j+k}(\omega^{k-1}-\omega^{k+1})\left(\rho^{n-h-1}_{j+k+1}-\rho^{n-h-1}_{j+k}\right)\\
    &\qquad\qquad\qquad\qquad-\omega^0\mathcal{R}^n_{j-1}\left(\rho^{n-h-1}_j-\rho^{n-h-1}_{j-1}\right)-\omega^1\mathcal{R}^n_j\left(\rho^{n-h-1}_{j+1}-\rho^{n-h-1}_j\right)\Big]\\
    &-\lambda\Delta x~v'(\xi^{n-h-1}_j)\Big[\sum_{k=1}^N\mathcal{L}^n_{j+k}(\omega^{k-1}-\omega^{k+1})\left(\rho^{n-h-1}_{j+k}-\rho^{n-h-1}_{j+k-1}\right)\\
    &\qquad\qquad\qquad\qquad-\omega^0\mathcal{L}^n_{j-1}\left(\rho^{n-h-1}_{j-1}-\rho^{n-h-1}_{j-2}\right)-\omega^1\mathcal{L}^n_j\left(\rho^{n-h-1}_j-\rho^{n-h-1}_{j-1}\right)\Big].
\end{align*}
Since $|\mathcal{R}^n_j|,|\mathcal{L}^n_j|\leq\frac{1}{2}[\alpha+(1+R\norma{f'})V]$ for every $j\in\Z$, and since
\begin{align*}
    \lambda\Delta x\Delta t\sum_{k=1}^N&(\omega^{k-1}-\omega^{k+1})\sum_{n=1}^{N_T-1}\sum_j\modulo{\rho^{n-h-1}_{j+k+1}-\rho^{n-h-1}_{j+k}}\phi^{n-1}_j\\
    \leq&\ 2\lambda\Delta x\Delta t\omega^0\norma{\phi}\sum_{n=1}^{N_T-1}\sum_{j=j_0}^{j_1+N}\modulo{\rho^{n-h-1}_{j+1}-\rho^{n-h-1}_j}\\
    \leq&\ 2\lambda\norma{\omega}\norma{\phi}\int_0^T\int_{-X}^{X+L}\modulo{\rho^{\Delta x}\left(t-(h+1)\Delta t,x+\Delta x\right)-\rho^{\Delta x}\left(t-(h+1)\Delta t,x-\Delta x\right)}\d x\d t\\
    \leq&\ 2\norma{\omega}\norma{\phi}\mathcal{C}\tv(\rho^0)\Delta t,
\end{align*}
where the positive constant $\mathcal{C}$ is given by Proposition \ref{BVteo}, then from \eqref{stimapart5} we get
\begin{align*}
    \eqref{part5}=\mathcal{O}(\Delta x+\Delta t),
\end{align*}
and this clearly proves that \eqref{part5} converges to zero as $\Delta x\rightarrow 0$ ( and $\Delta t\rightarrow 0$).
\end{proof}

By properly adapting the doubling of variables technique due to Kru{\v{z}}kov, now we prove a uniqueness result within the class of entropy solutions for the initial value problem.
The uniqueness follows from the Lipschitz continuous dependence of entropy solutions with respect to initial data, that is asserted in the following Theorem. 

\begin{theorem}[$\L1$ stability]\label{stabilityteo}
    Given Assumption \ref{hp}, let $\rho$ and $\sigma$ be two entropy weak solutions of \eqref{delay_bis}-\eqref{eq:initial datum} as in Definition \ref{entropy}, with initial data $\rho^0, \sigma^0 \in \BV(\R;[0,R])$ and time delay parameters $\tau_1, \tau_2 >0$, respectively. Then, for any $T>0$ there holds
    \begin{equation}\label{stimaL1}
        \norma{\rho(t,\cdot)-\sigma(t,\cdot)}_{1}\leq e^{K_1T}
        \left(K_3\norma{\rho^0-\sigma^0}_{1}
        + K_2\modulo{\tau_1 - \tau_2}
        \right),
        \qquad\forall t\in[0,T],
    \end{equation}
    with $K_1>0$ given by \eqref{K1}, $K_2>0$ by \eqref{K2} and $K_3>0$ by \eqref{K3}.
\end{theorem}

\begin{proof}
    In this proof, we follow the steps of  \cite[Lemma 4 and Proof of Theorem 1]{ChiarelloGoatinRossi2019}.
    The functions $\rho$ and $\sigma$ are respectively entropy solutions of the equations
    \begin{eqnarray*}
        \partial_t\rho(t,x)+\partial_x\big(\rho(t,x)f(\rho(t,x))\mathcal{V}(t-\tau_1,x)\big)=0,&&\mathcal{V}(t,x):=v((\rho\ast\omega)(t,x)),\\
        \partial_t\sigma(t,x)+\partial_x\big(\sigma(t,x)f(\sigma(t,x))\mathcal{U}(t-\tau_2,x)\big)=0,&&\mathcal{U}(t,x):=v((\sigma\ast\omega)(t,x)),
    \end{eqnarray*}
    and they fulfill the following initial conditions
    \begin{align*}
        \rho(t,x)&=\rho^0(x),\qquad\mbox{ for }(t,x)\in[-\tau_1,0]\times\R,\\
        \sigma(t,x)&=\sigma^0(x),\qquad\mbox{ for }(t,x)\in[-\tau_2,0]\times\R.
    \end{align*}
    Since $\rho,\sigma\in\L1([0,T]\times\R,[0,R])$, $\mathcal{V}$ and $\mathcal{U}$ are bounded measurable functions and are Lipschitz continuous w.r.t. $x$. Thus, we have
    $$
    \norma{\mathcal{V}_x}\leq 2\norma{\omega}\norma{v'}\norma{\rho}\qquad,\qquad\norma{\mathcal{U}_x}\leq 2\norma{\omega}\norma{v'}\norma{\sigma},
    $$
     where we recall that we are using the compact notation $\norma{\cdot}$ for $\norma{\cdot}_{\infty}$.
    Using Lemma \ref{boundteo} and Kru{\v{z}}kov's doubling of variables technique, we obtain the following inequality:
    \begin{align}
        \norma{\rho(T,\cdot)-\sigma(T,\cdot)}_1\leq&\norma{\rho^0-\sigma^0}_{1}\label{normaL1}\\
        &+\left(1+R\norma{f'}\right)\int_0^T\int_\R\modulo{\del_x\rho(t,x)} \modulo{\mathcal{V}(t-\tau_1,x)-\mathcal{U}(t-\tau_2,x)}\d x\d t\nonumber\\
        &+R\int_0^T\int_\R\modulo{f(\rho(t,x))}\modulo{\del_x\mathcal{V}(t-\tau_1,x)-\del_x\mathcal{U}(t-\tau_2,x)}\d x\d t.\nonumber
    \end{align}
    First, we can bound
    \begin{equation}\label{normaL1_1}
    \modulo{\mathcal{V}(t-\tau_1,x)-\mathcal{U}(t-\tau_2,x)}\leq\norma{\omega}\norma{v'}\norma{\rho(t-\tau_1,\cdot)-\sigma(t-\tau_2,\cdot)}_{1}.
    \end{equation}
    Moreover, Lemma \ref{L1boundteo} ensures
    \begin{align}
        \int_\R&\modulo{\del_x\mathcal{V}(t-\tau_1,x)-\del_x\mathcal{U}(t-\tau_2,x)}\d x=\nonumber\\
        =&\int_\R\modulo{v'((\rho\ast\omega)(t-\tau_1,x))(\rho\ast\partial_x\omega)(t-\tau_1,x)-v'((\sigma\ast\omega)(t-\tau_2,x))(\sigma\ast\partial_x\omega)(t-\tau_2,x)}\d x\nonumber\\
        \leq&\ \norma{v'}\int_\R\modulo{(\rho\ast\partial_x\omega)(t-\tau_1,x)-(\sigma\ast\partial_x\omega)(t-\tau_2,x)}\d x\nonumber\\
        &+\norma{v''}\int_\R\modulo{(\rho\ast\omega)(t-\tau_1,x)-(\sigma\ast\omega)(t-\tau_2,x)} \modulo{(\sigma\ast\partial_x\omega)(t-\tau_2,x)}\d x\nonumber\\
        \leq&\  \norma{v'}\norma{\partial_x\omega}_{1}\norma{\rho(t-\tau_1,\cdot)-\sigma(t-\tau_2,\cdot)}_{1}\nonumber\\
        &+\norma{v''} \norma{\sigma(t-\tau_2,\cdot)}_{1}\norma{\partial_x\omega} \norma{\omega}_{1}
        \norma{\rho(t-\tau_1,\cdot)-\sigma(t-\tau_2,\cdot)}_{1}\nonumber\\
        =& \left(\norma{v'}\norma{\partial_x\omega}_1+\norma{v''} \norma{\sigma^0}_{1} \norma{\partial_x\omega}\norma{\omega}_1\right)\norma{\rho(t-\tau_1,\cdot)-\sigma(t-\tau_2,\cdot)}_{1}.\label{normaL1_2}
    \end{align}
We suppose without loss of generality that $\tau_1\geq\tau_2$.
Thus, plugging \eqref{normaL1_1} and \eqref{normaL1_2} into \eqref{normaL1}, we get
    \begin{align}
        \norma{\rho(T,\cdot)-\sigma(T,\cdot)}_{1}\leq&\ \norma{\rho^0-\sigma^0}_{1}
        +K_1\int_0^T\norma{\rho(t-\tau_1,\cdot)-\sigma(t-\tau_2,\cdot)}_{1}\d t\nonumber\\
        \leq&\ \norma{\rho^0-\sigma^0}_{1}\nonumber \\
        &+K_1\int_0^T\norma{\rho(t-\tau_2,\cdot)-\sigma(t-\tau_2,\cdot)}_{1}\d t\label{tau2tau2}\\
        &+K_1\int_0^T\norma{\rho(t-\tau_1,\cdot)-\rho(t-\tau_2,\cdot)}_{1}\d t\label{tau1tau2},
        \end{align}
    being
    \begin{align}
        K_1=&\ \norma{\omega}\norma{v'}\left(1+R\norma{f'}\right)\sup_{t\in[0,T]}\norma{\rho(t,\cdot)}_{\BV(\R)}
        +R
        \left(\norma{v'}\norma{\partial_x\omega}_1+\norma{v''} \norma{\sigma^0}_{1} \norma{\partial_x\omega}\norma{\omega}_1\right)\label{K1}.
    \end{align}
    We remark that from Remark \ref{relazioneoss} it follows that $K_1$ is dimensionally the inverse of time.
    On the one hand,
    \begin{align*}
     \eqref{tau2tau2}=K_1\int_{-\tau_2}^{T-\tau_2}\norma{\rho(t,\cdot)-\sigma(t,\cdot)}_{1}\d t\leq K_1\tau_2\norma{\rho^0-\sigma^0}_{1}+K_1\int_{0}^{T}\norma{\rho(t,\cdot)-\sigma(t,\cdot)}_{1}\d t,
    \end{align*}
where we used the initial conditions and the fact that $\rho$ and $\sigma$ are weak solutions, and so they are defined for every time horizon $[0,T]$ because of Theorem \ref{E1}.
On the other hand, using the $\L1$ Lipschitz continuity in time stated in~\eqref{eq:L1time},
we get
\begin{align*}
    \eqref{tau1tau2} \leq K_2 \modulo{\tau_1-\tau_2},
\end{align*}
with
\begin{equation}\label{K2}
K_2 =     K_1 \mathcal{K} T,
\end{equation}
that has the same dimension of $\mathcal{K}$ that is the inverse of time, as we can understand from \eqref{k LF}.
Thus,
$$
\norma{\rho(T,\cdot)-\sigma(T,\cdot)}_{1}\leq K_3\norma{\rho^0-\sigma^0}_{1}+K_2 \modulo{\tau_1-\tau_2}+K_1\int_0^T\norma{\rho(t,\cdot)-\sigma(t,\cdot)}_{1}\d t,
$$
where we have introduced the dimensionless constant
\begin{equation}\label{K3}
    K_3=1+K_1\min\{\tau_1,\tau_2\}.
\end{equation}
The statement follows from Gronwall's lemma.
\end{proof}

Theorem \ref{stabilityteo} also proves a stability property w.r.t. the delay $\tau$, that we state as a corollary.
A similar convergence result was proved in \cite[Theorem 4.3]{KeimerPflug2019} for the delayed model \eqref{delay}. We remark that our result is stronger since it holds for every time horizon $T>0$.

\begin{cor}[Convergence for delay tending to zero]\label{stabilitycor}
    Let Assumption \ref{hp} hold. Given $\rho^0\in\BV(\R;[0,R])$, let $\rho_\tau\in\L1([0,T]\times\R;\R)$ denote the solution of the Cauchy problem \eqref{delay_bis}-\eqref{eq:initial datum} for $\tau>0$. Let also $\rho\in\L1([0,T]\times\R;\R)$ be the entropy solution of
    \begin{equation}\label{nodelay_bis}
    \partial_t \rho(t,x)+\partial_x\big(\rho(t,x)f(\rho(t,x))v((\rho\ast \omega)(t,x))\big)=0,
    \end{equation}
    with the same initial condition, as defined in \cite[Definition 1.1]{ChiarelloGoatin2018}.
    Then, we have the convergence
    $$
    \norma{\rho_\tau(t,\cdot)-\rho(t,\cdot)}_{\L1}\rightarrow 0,\qquad\forall t\in[0,T],
    $$
    as $\tau\searrow 0$.
\end{cor}

\begin{proof}
    The statement is a direct consequence of Theorem \ref{stabilityteo} and \cite[Theorem 1.2]{ChiarelloGoatin2018}.
\end{proof}

\section{Numerical tests} \label{sec:num}
This section is devoted to present some numerical results concerning the nonlocal model with time delay~\eqref{delay_bis}.

\subsection{Comparison between HW and Lax-Friedrichs}\label{seconfrontoschemi}

In \cite{2018Gottlich}, the authors show the advantages of the HW scheme \eqref{schema2} with respect to the widely used LF scheme \eqref{schema}. They underline that one of the major advantages of HW, unlike LF, is that in \eqref{schema2} the numerical fluxes are always non-negative. Furthermore, HW is known to be less diffusive, as confirmed by the following experiments. 

Let us consider the space domain $[0,1]$ with free-flow boundary conditions and the following Riemann-like initial conditions
\begin{equation}\label{eq:ICtest1}
\rho^0(x)=
\left\{
\begin{array}{ll}
0.3&\mbox{  if }x < 0.5,\\
1.5&\mbox{  if }x \geq 0.5,
\end{array}
\right.
\quad\mbox{ and }\quad
\rho^0(x)=
\left\{
\begin{array}{ll}
1.5&\mbox{  if }x\leq 0.5,\\
0.3&\mbox{  if }x>0.5,
\end{array}
\right.
\end{equation}
that correspond respectively to a shock and a rarefaction wave for the classical (i.e. local) non delayed equation.
For simplicity, we choose the Greenshields' velocity function
\begin{equation}\label{Green}
v(\rho)=V\left(1-\frac{\rho}{R}\right),\qquad\mbox{ with }\quad V=0.9\quad\mbox{ and }\quad R=1.7,
\end{equation}
while the saturation is given by the linear function 
\begin{equation}\label{flinear}
    f(\rho)=1-\frac{\rho}{R}.
\end{equation}
We fix the look-ahead distance $L=0.015$, the time delay parameter $\tau=10^{-2}$ and we consider a constant weight  kernel $\omega(s)=\frac{1}{L}$.
In Figure \ref{confrontoschemi} we compare the numerical solutions at the final time $T=0.5$ obtained via LF and HW with $\Delta x=5\cdot 10^{-3}$ as the space step and the time step $\Delta t$ given by the CFL condition. We also represent a reference solution, which is computed with LF and the small space step $\Delta x=2.5\cdot 10^{-4}$. We can clearly see that HW is closer to the reference solution in both cases. 
\\In the rest of the section, we therefore perform all numerical tests using the HW scheme.

\begin{figure}[ht]
\centering
{\includegraphics[width=0.45\textwidth]{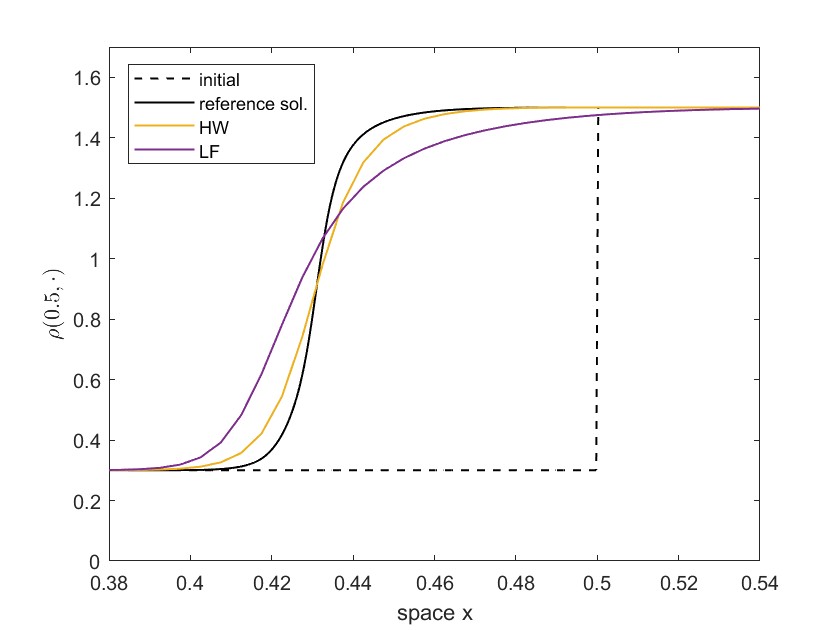}}
{\includegraphics[width=0.45\textwidth]{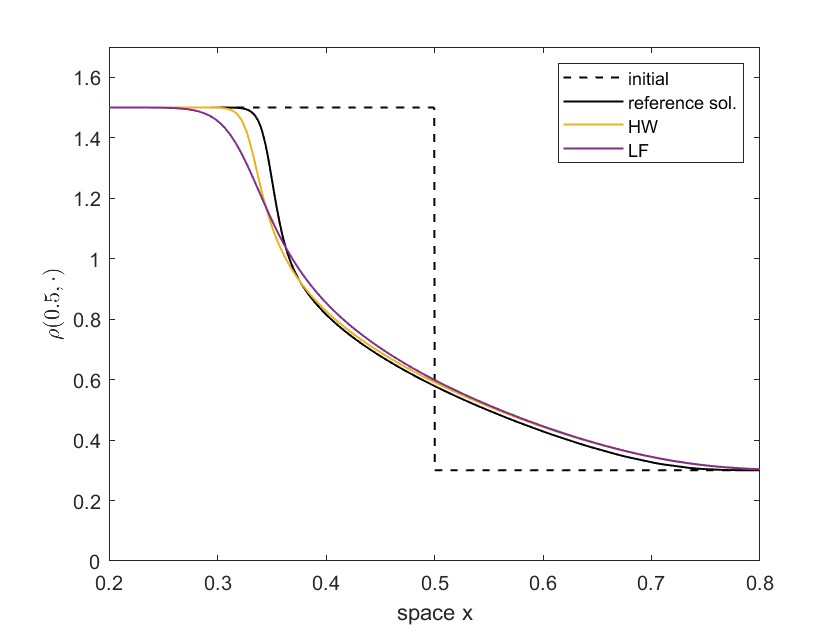}}
\caption{Comparison between LF \eqref{schema} and  HW \eqref{schema2} schemes  for $\Delta x=5\cdot 10^{-3}$ corresponding to initial data~\eqref{eq:ICtest1}.}\label{confrontoschemi}
\end{figure}

\subsection{Space step tending to zero}
As a consequence of the $\BV$ estimate \eqref{eq:TV}, the entropy solution of the model \eqref{delay_bis} may develop oscillations due to the delay in time. Obviously, we expect the numerical solution to be able to capture such behavior. In the following, we report a study of the solution $\rho^{\Delta x}$ for fixed values of $\tau=0.1$ and $L=0.15$, and for $\Delta x\searrow 0$. We consider the space domain $[0,5]$ with again free-flow boundary conditions, and we choose the Greenshields' velocity function \eqref{Green} as above.
Instead of the linear function~\eqref{flinear}, let us consider the family of exponential saturation functions 
\begin{equation}\label{fexp}
    f_\varepsilon(\rho)=1-e^{(\rho-R)/\epsilon},\qquad\mbox{ for }\varepsilon>0.
\end{equation}
The above sequence converges quasi-uniformly to $\chi_{[0,R[}$ as $\varepsilon\searrow 0$, but $f_\varepsilon (R)=0$ for all $\varepsilon$. Anyway, we cannot let $\varepsilon$ go to zero, since we need  $\norma{f'}$ to be bounded for the convergence of the scheme to hold (see Assumption \ref{hp}).
Choosing the saturation term as in~\eqref{fexp}, with for instance $\varepsilon=1/50$, model~\eqref{delay_bis} is expected to behave as \eqref{delay} for density values sufficiently smaller than the maximal density. \\
We consider the initial datum\begin{equation}\label{datum}
    \rho^0(x)=
    \left\{
    \begin{array}{ll}
    1.5&\mbox{ if }1\leq x\leq 2,\\
    0&\mbox{ otherwise,}
    \end{array}
    \right.
\end{equation}
and we fix a linear decreasing weight kernel $\omega(x)=\frac{2}{L}\left(1-\frac{x}{L}\right)$, which is a more reasonable choice to describe human behavior with respect to the constant one, as we underlined in Remark \ref{hpremark}. In Figure \ref{confrontodelta} we compare the numerical solution at the final time $T=0.5$ computed with decreasing values of the space step $\Delta x$ and with a time step $\Delta t$ chosen in such a way that the CFL condition \eqref{CFL HW} is fulfilled. 
We can see that, as $\Delta x$ diminish, the amplitudes (not the number) of the oscillations increase and the numerical scheme is able to better capture the properties of the entropy solution. We remark that the density cannot overtake the value $R$, which means that the oscillations can amplify only up to that value.

\begin{figure}
\centering
{\includegraphics[width=12cm]{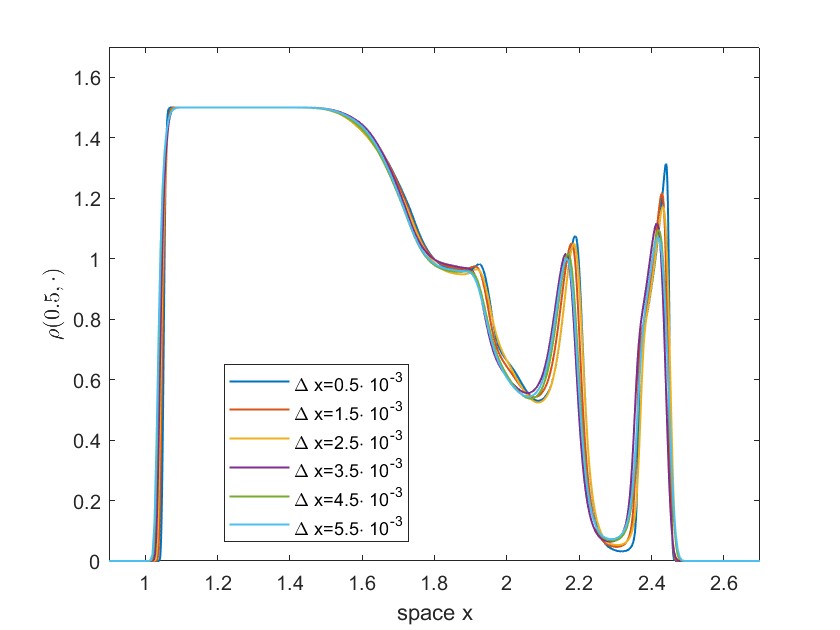}}
\caption{Comparison between the solution to the model \eqref{delay_bis} associated to the initial datum $\rho^0(x)=3/2\chi_{[1,2]}(x)$ and to the parameters $\tau=0.1$ and $L=0.15$, with decreasing values of the space step $\Delta x$.}\label{confrontodelta}
\end{figure}

\subsection{The effect of the saturation}\label{saturationsec}

We recall that we introduced the saturation function in order to guarantee that solutions do not exceed the maximum density $R$. Indeed, as noted in~\cite{KeimerPflug2019}, the solution to equation~\eqref{delay} can exceed $R$, and thus the velocity could be negative. To avoid this, in their tests the authors have chosen the cropped velocity function $v(\rho)=1-\min\{\rho,1\}$. With this choice, the solution to the delayed model \eqref{delay_bis} appears to be much smoother than the solution to the classical delayed model \eqref{delay}, for every choice of the time delay parameter $\tau$, but not exceeding $R$. 
\begin{figure}
\centering
{\includegraphics[width=15cm]{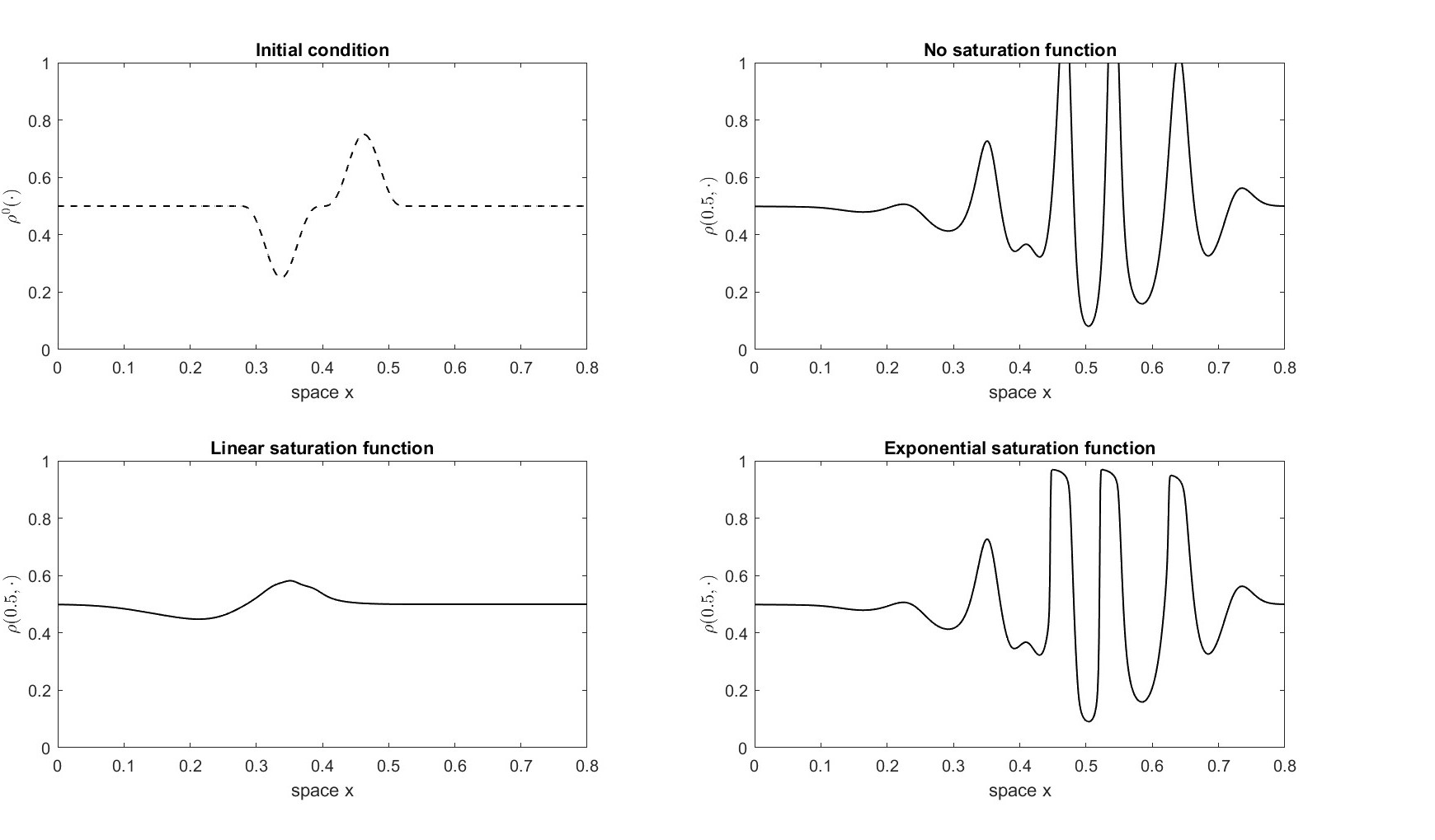}}
{\includegraphics[width=15cm]{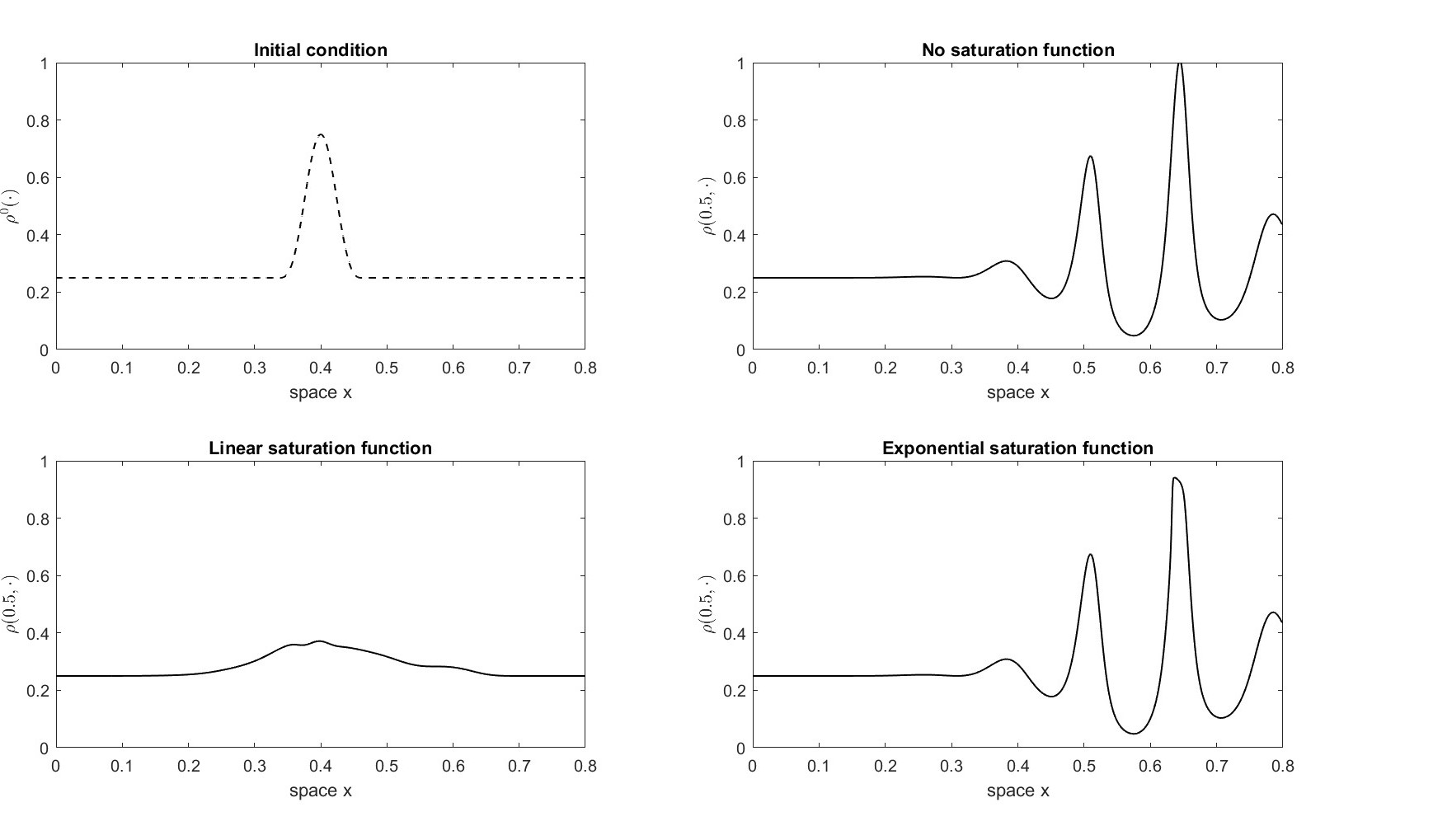}}
\caption{Comparison between the solution to the model with no saturation \eqref{delay}, and  the solution to the model \eqref{delay_bis} with $\tau=0.12$, and with saturation functions \eqref{flinear}\eqref{fexp} .\textbf{ Two top rows}: initial datum \eqref{eq:ICtest24}; \textbf{Two bottom rows}: initial datum \eqref{eq:ICtest23} with $\bar{\rho}=1/4$.}\label{confrontof23}
\end{figure}
\begin{figure}
\centering
{\includegraphics[width=15cm]{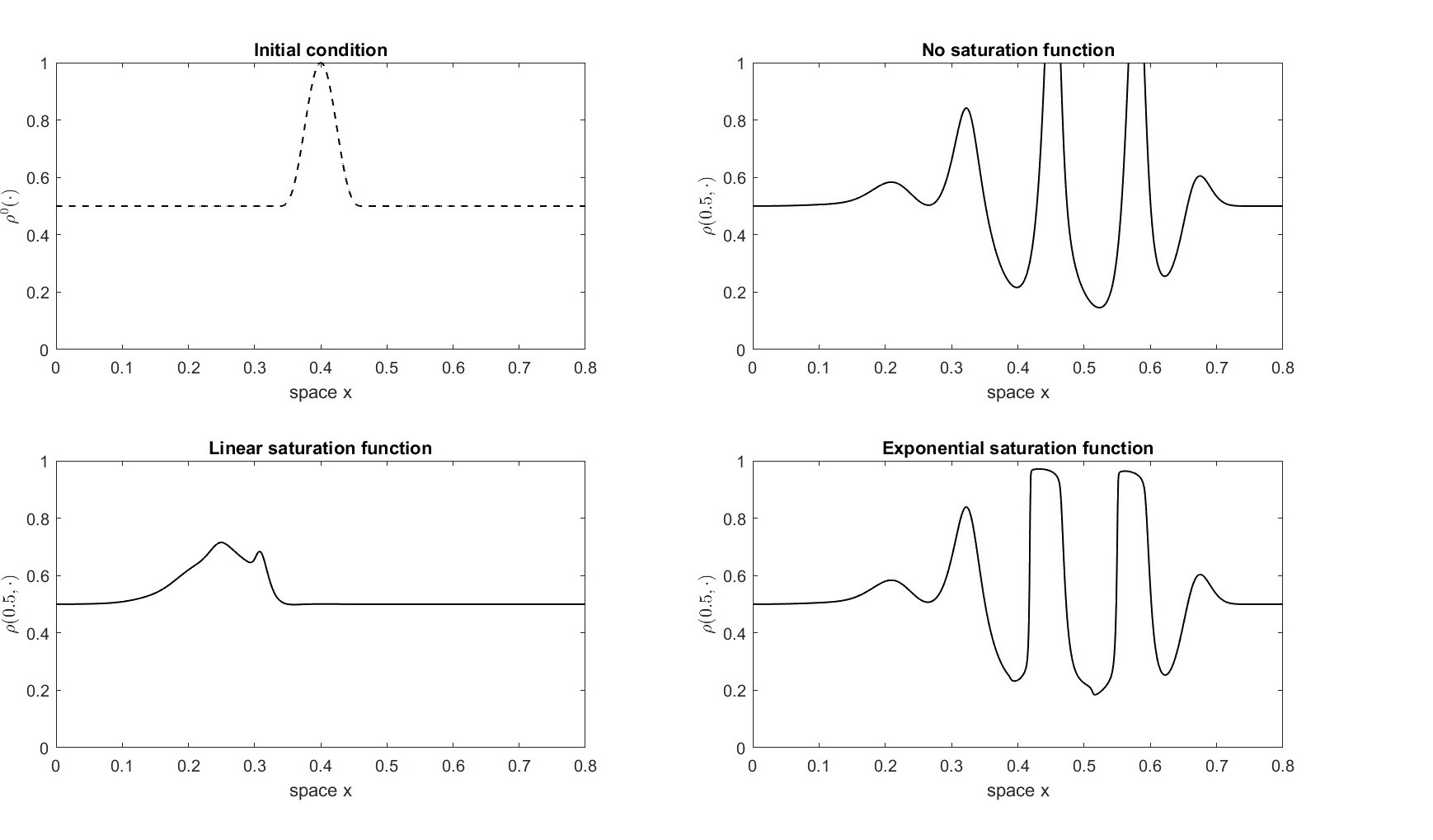}}
\caption{Comparison between the solution to the model with no saturation \eqref{delay}, and  the solution to the model \eqref{delay_bis} with saturation functions \eqref{flinear}\eqref{fexp}, associated to the initial datum \eqref{eq:ICtest23} with $\bar{\rho}=1/2$ and to the delay $\tau=0.08$.}\label{confrontof4}
\end{figure}
In Figure~\ref{confrontof23}-\ref{confrontof4}, we consider the initial conditions
\begin{align}
\rho^0(x)&=\frac{1}{2}+\left[\frac{3}{16}\sin\Big(8\pi\big(x-\frac{2}{5}\big)\Big)-\frac{1}{16}\sin\Big(24\pi\big(x-\frac{2}{5}\big)\Big)\right]\chi_{\left[\frac{11}{40},\frac{21}{40}\right]}(x),
\label{eq:ICtest24}
\\
\rho^0(x)&=\bar{\rho}+\left[\frac{3}{8}\cos\Big(8\pi\big(x-\frac{2}{5}\big)\Big)+\frac{1}{8}\cos\Big(24\pi\big(x-\frac{2}{5}\big)\Big)\right]\chi_{\left[\frac{27}{80},\frac{37}{80}\right]}(x),\quad\bar{\rho}\in\left\{\frac{1}{4},\frac{1}{2}\right\}.
\label{eq:ICtest23}
\end{align}
We choose again $\Delta x=10^{-3}$ and we fix $L=0.1$ and we compare the solutions at the final time $T=0.5$ of the classical delayed model \eqref{delay} and our delayed model~\eqref{delay_bis} associated to the linear saturation function~\eqref{flinear} and the exponential one~\eqref{fexp} with $\epsilon=1/50$.
\\In all the cases, we consider the velocity function 
\begin{equation}\label{Greennorm}
v(\rho)=1-\rho
\end{equation}
( i.e. $V=R=1$) and a constant weight kernel $\omega(x)=1/L$ for the sake of simplicity. We remark that the presence of a saturation term guarantees that the solution is bounded by $R$ even if the velocity is not cropped, consistently with Theorem \ref{boundteo}. This is true in both the cases of a linear and an exponential decreasing saturation function. When choosing an $f$ as in~\eqref{fexp}, the stabilizing effect is way less visible, and the density appears to be just a bounded form of the solution to the model with no saturation.

\subsection{Convergence to the non-delayed model}\label{convergencesec}

In the following examples, we illustrate the convergence of the solution for delay tending to zero, which is stated in Corollary~\ref{stabilitycor}. 
We consider the space domain $[0,5]$ with free-flow boundary conditions and we fix the space step $\Delta x=5\cdot 10^{-3}$. In Figure \ref{limitdelayfigure},  we compare the numerical solutions of the delayed equation~\eqref{delay_bis} associated to a constant look-ahead distance $L=0.15$ and different values of the time delay parameter $\tau$, to the solution of the non-delayed equation~\eqref{nodelay_bis} at the same final time $T=0.5$ and corresponding to the same initial data, velocity function \eqref{Greennorm}, linear decreasing weight kernel $\omega(x)=\frac{2}{L}\left(1-\frac{x}{L}\right)$ and the exponential saturation function~\eqref{fexp} with again $\epsilon=1/50$:
\begin{figure}[ht]
\centering
{\includegraphics[width=0.45\textwidth]{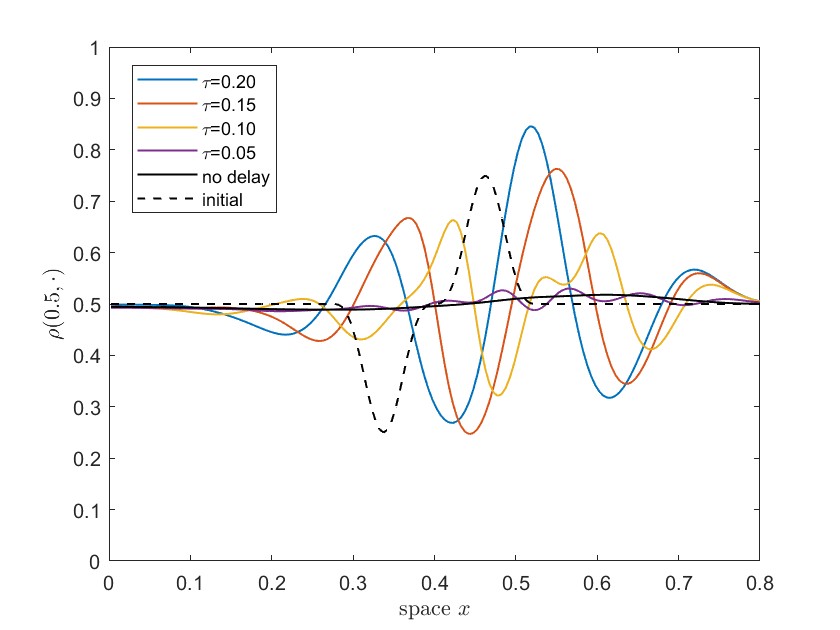}}
{\includegraphics[width=0.45\textwidth]{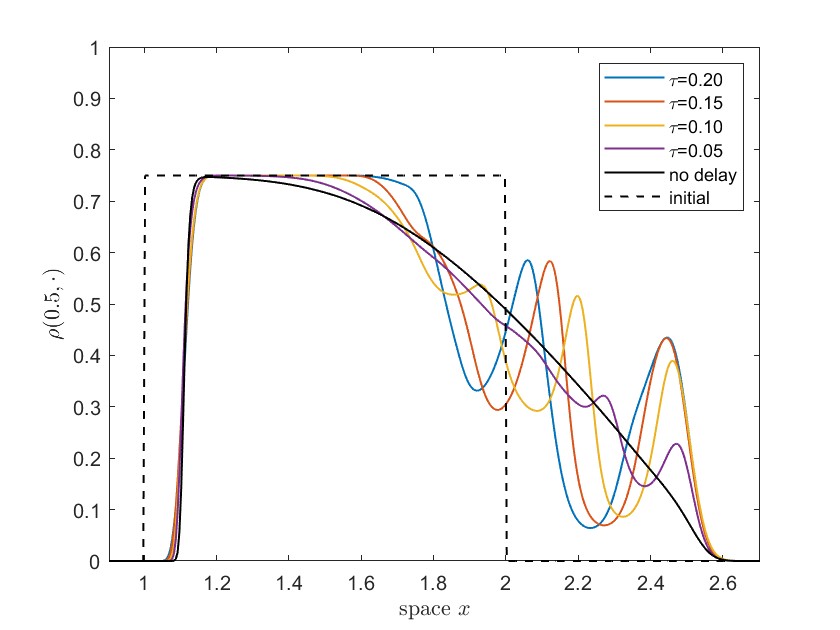}}
\caption{Convergence of the delayed model \eqref{delay_bis} to the non-delayed one \eqref{delay} with velocity \eqref{Greennorm}, linear decreasing kernel and exponential saturation function \eqref{fexp} with $\epsilon=1/50$, as the time delay parameter $\tau$ approaches zero. \textbf{Left:} initial datum \eqref{eq:ICtest21}. \textbf{Right:} initial condition~\eqref{datumnorm}.}\label{limitdelayfigure}
\end{figure}
\begin{figure}
\centering
{\includegraphics[width=0.45\textwidth]{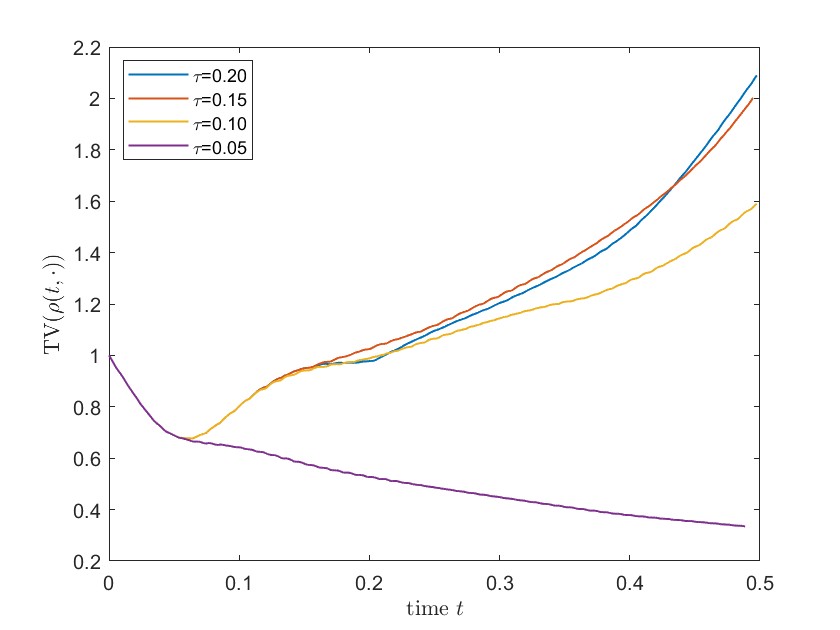}}
{\includegraphics[width=0.45\textwidth]{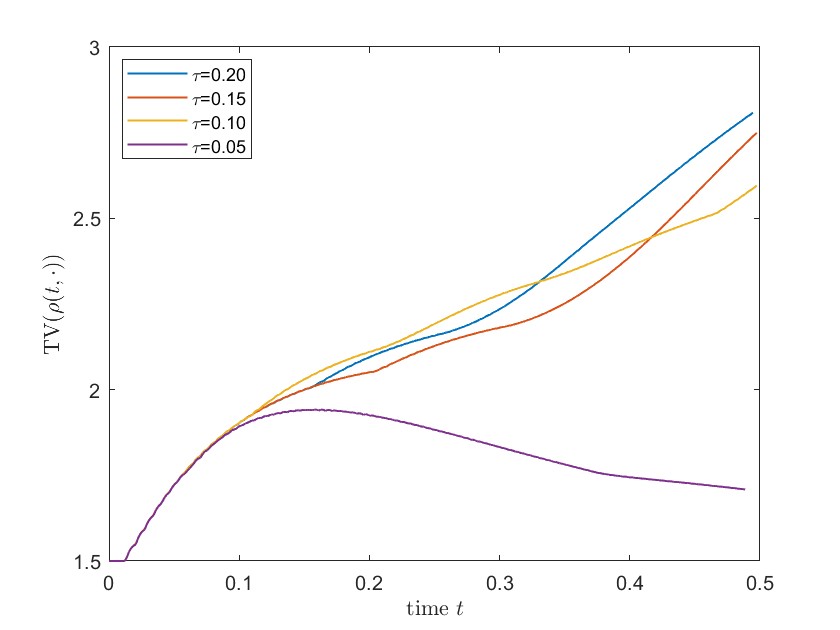}}
\caption{Total variation of the tests in Figure \ref{limitdelayfigure}. \textbf{Left:} initial datum \eqref{eq:ICtest21}. \textbf{Right:} initial condition~\eqref{datumnorm}.
}\label{totalvariationfigure}
\end{figure}

\begin{figure}
\centering
{\includegraphics[width=15cm]{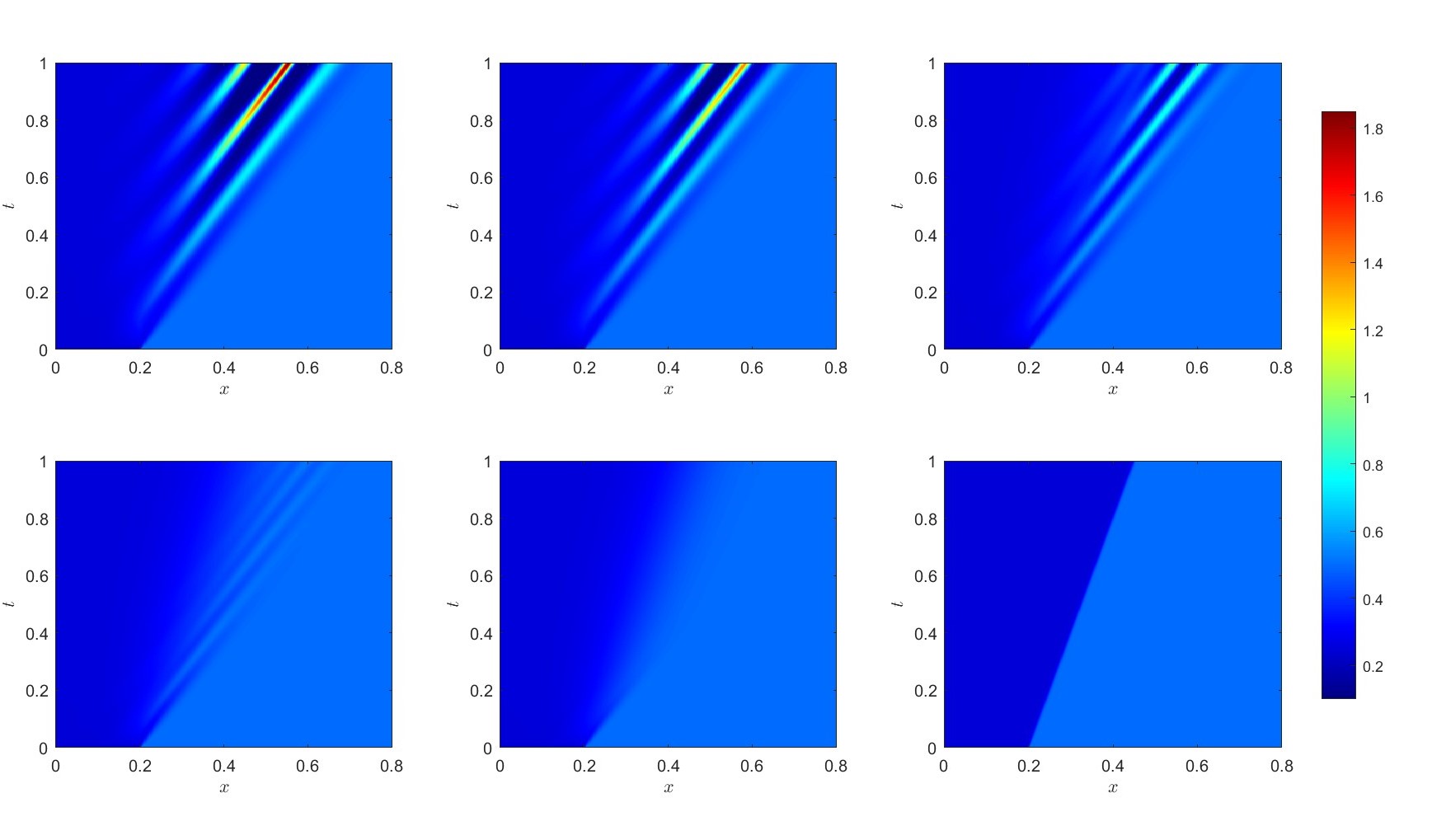}}
{\includegraphics[width=15cm]{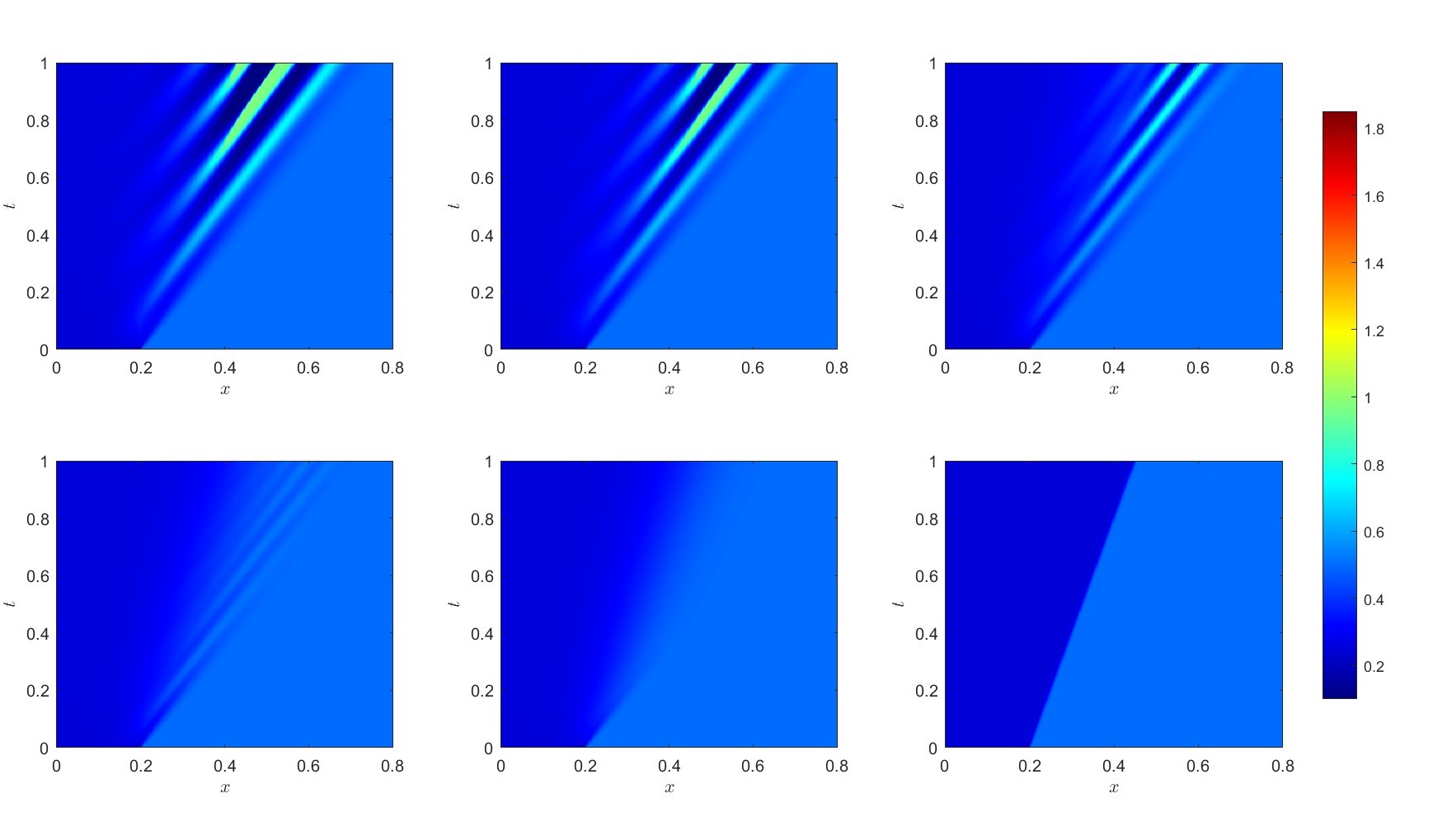}}
\caption{\textbf{Two top rows}: Keimer and Pflug's delayed model \eqref{delay}, $\tau=0.1,0.08,0.06,0.04,0.02$ (as in~\cite[Figure 3]{KeimerPflug2019}); \textbf{Two bottom rows}: delayed model \eqref{delay_bis}, same $\tau$ and exponential saturation function. The lower rightmost figure in both the couples of rows is the solution to the LWR model with initial datum \eqref{eq:ICtest22}. For the other graphs, we fix $L=0.1$  and linear decreasing kernel.}\label{keimer1}
\end{figure}
\begin{figure}
\centering
{\includegraphics[width=15cm]{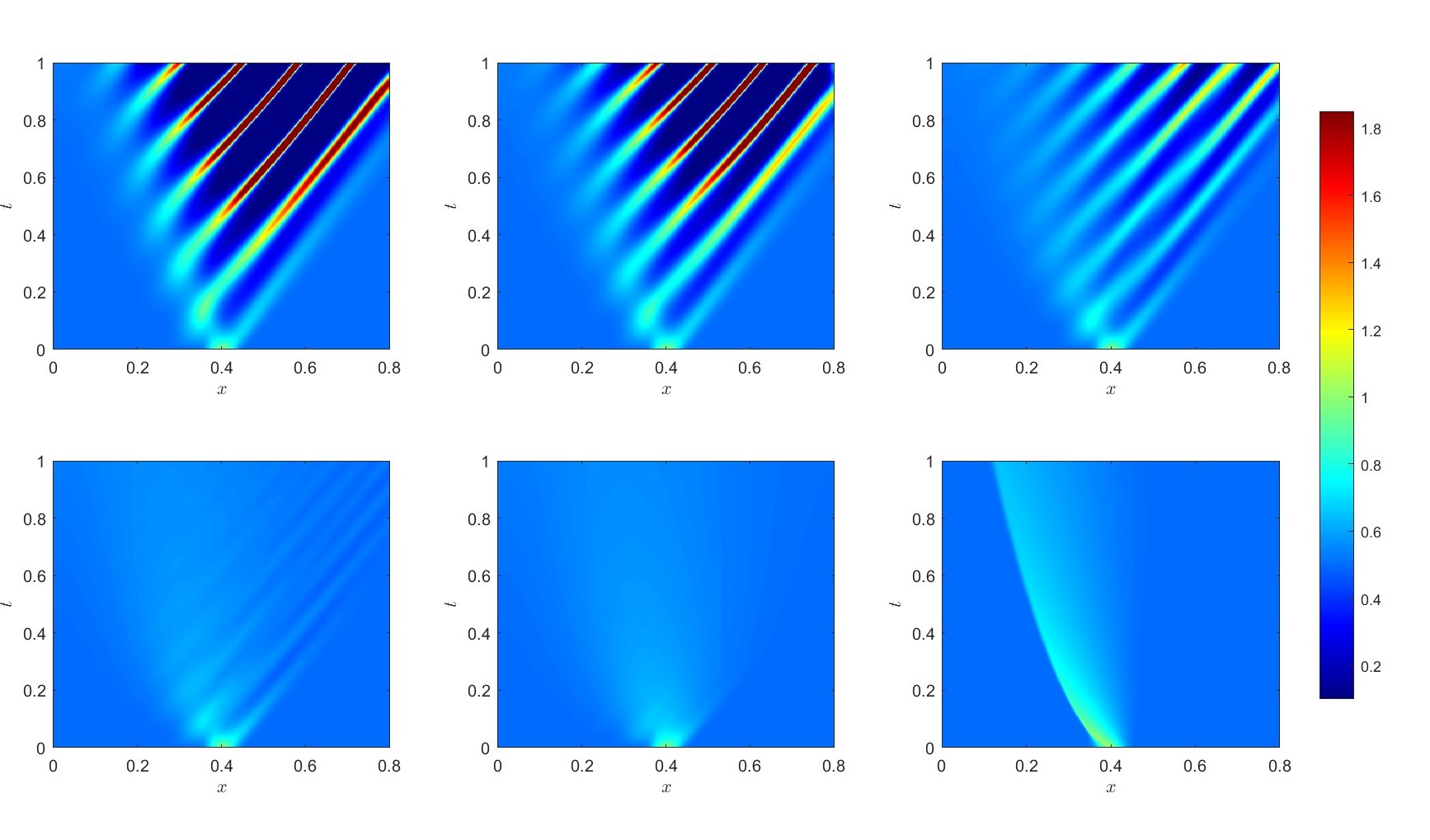}}
{\includegraphics[width=15cm]{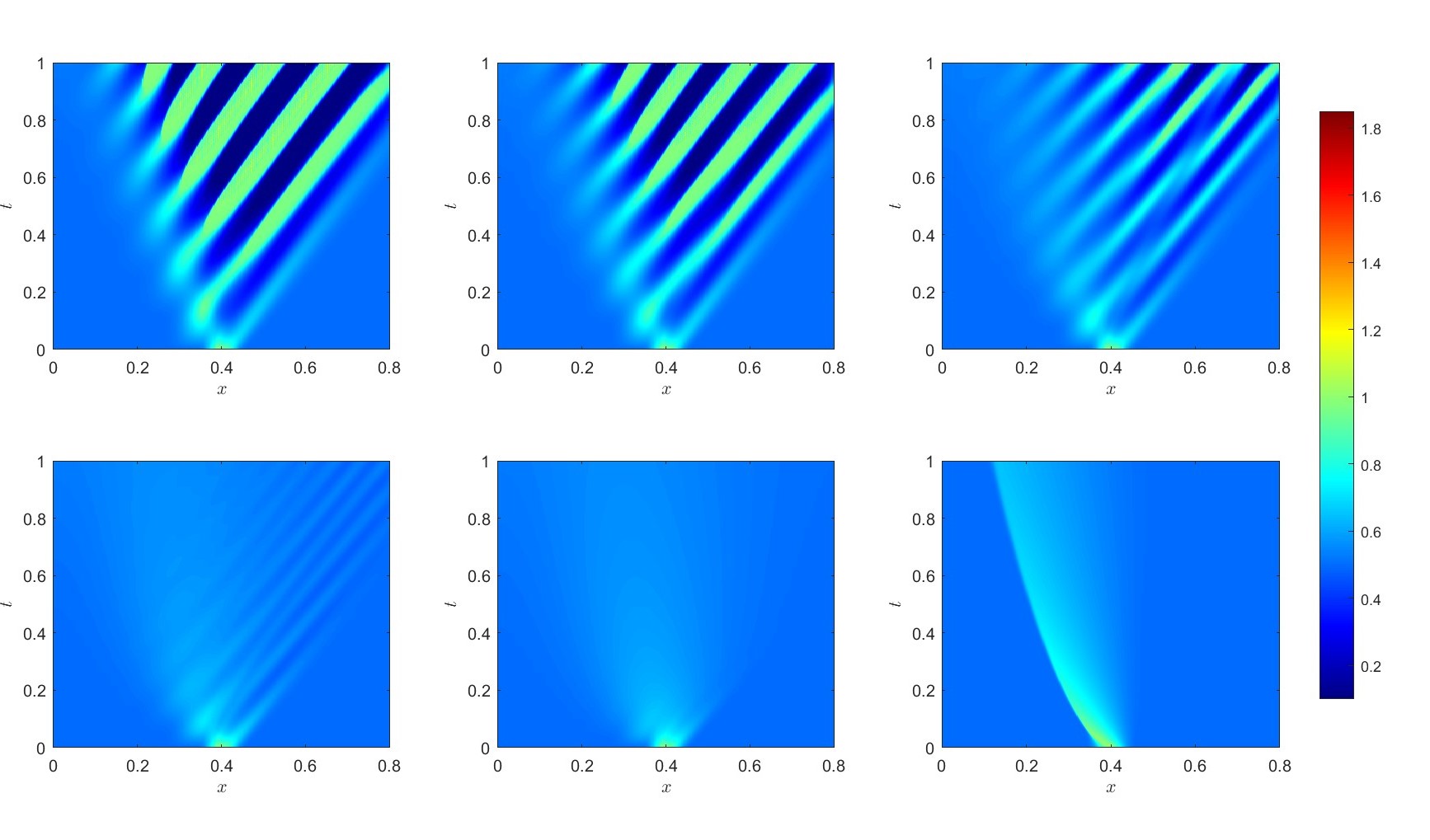}}
\caption{\textbf{Two top rows}: Keimer and Pflug's delayed model \eqref{delay}, $\tau=0.1,0.08,0.06,0.04,0.02$ (as in~\cite[Figure 5]{KeimerPflug2019}); \textbf{Two bottom rows}: delayed model \eqref{delay_bis}, same $\tau$ and exponential saturation function. The lower rightmost figure in both the couples of rows is the solution to the LWR model with initial datum \eqref{eq:ICtest23} with $\bar{\rho}=1/2$. For the other graphs, we fix $L=0.1$  and linear decreasing kernel.}\label{keimer3}
\end{figure}
\begin{itemize}
    \item 
On the left, following \cite[Section 5]{KeimerPflug2019}, we consider the initial datum 
\begin{equation} \label{eq:ICtest21}
    \rho^0(x)=\frac{1}{2}+\left[\frac{3}{16}\sin\left(8\pi\big(x-\frac{1}{2}\big)\right)-\frac{1}{16}\sin\left(24\pi\big(x-\frac{1}{2}\big)\right)\right]\chi_{\left[\frac{11}{40},\frac{21}{40}\right]}(x),
\end{equation}
 which is chosen in order to show how a small oscillation in the initial condition affects the evolution depending on the delay.  

\item On the right, we choose the initial condition 
\begin{equation}\label{datumnorm}
    \rho^0(x)=
    \left\{
    \begin{array}{ll}
    3/4&\mbox{ if }1\leq x\leq 2,\\
    0&\mbox{ otherwise,}
    \end{array}
    \right.
\end{equation}
which is a normalized version of \eqref{datum}.
\end{itemize}

Our numerical simulations are consistent with what it is observed in~\cite{KeimerPflug2019}:  the oscillations increase as $\tau$ increases, denoting higher total variation in accordance with~\eqref{eq:TV}. As a further proof, in Figure \ref{totalvariationfigure} we graph the total variation of the solution of the tests in Figure \ref{limitdelayfigure} as a function of time, computed with the different values of the time delay parameter. We see again higher values of the total variation for higher delays. 
\\The delay also influences the traffic flow by causing the formation of the so-called stop and go waves. In Figures~\ref{keimer1}-\ref{keimer3} we propose some examples that are inspired by Keimer and Pflug's numerical tests, in which this phenomena is clealy visible. In each of these figures, in the two top rows we show the numerical solution to their model~\eqref{delay}, while in the two bottom rows we plot the solutions to our model~\eqref{delay_bis}. In each couple of rows, the lower rightmost graph shows the solution to the classical LWR model associated to the Riemann-like initial datum
\begin{align}
\rho^0(x)&=
\left\{
\begin{array}{ll}
1/4&\mbox{ if }x\leq 0.2,\\
1/2&\mbox{ otherwise,}
\end{array}
\right.
\label{eq:ICtest22}
\end{align}
and to the same initial condition of Figure \ref{confrontof4}, which is given by \eqref{eq:ICtest23} with $\bar{\rho}=1/2$.
The other plots in the figures correspond to the numerical solutions associated to the same initial data, linear decreasing kernel with $L=0.1$, exponential saturation function \eqref{fexp} with again $\epsilon=1/50$ and different values of $\tau\in\{0.1,0.08,0.06,0.04,0.02\}$. The solutions are computed with the space step $\Delta x=2.3\cdot 10^{-3}$ in the space domain $[0,0.8]$.
We observe that the convergence to the non-delayed model can still be seen in the sense that when $\tau\rightarrow 0$ these clustering effects decrease in amplitude as well as the area on which they occur.
The graphs that we have just showed also clarify again how the saturation influences the solution. 

\section{Conclusions and perspectives}

In this paper, we introduced and studied a non-local macroscopic traffic flow model with time delay, which is able to model the nonzero reaction time that human-beings have in their response to a stimulus. We showed the well-posedness of the model through both Lax-Friedrichs and Hilliges-Weidlich  numerical schemes, analyzing the limit of the approximate solutions as the space step for the discretization tends to zero. We also proved a maximum principle and the stability with respect to both the delay and the initial condition, which provides the convergence to the limit model as the delay tends to zero. We recall that this property in general does not hold for both the local \cite{Gottlich2021} and the non-local \cite{KeimerPflug2019} existing models with time delay, meaning that in general $\rho>R$ is possible.
\\We then proposed some numerical simulations to illustrate our results. We showed that working with the Hilliges-Weidlich scheme is more convenient in terms of accuracy with respect to the Lax-Friedrichs scheme.
Additionally, the numerical experiments conducted in this study reveal that the model is able to exhibit stop-and-go waves, which seem to be a direct consequence of the delay in time. 

The above mentioned results open several perspectives for future research. First of all, it constitutes a very good basis to derive a multi-class model. 
Through this extension, we intend to further investigate the interactions among several classes of vehicles with different behaviors in terms of look-ahead distance and reaction time. 
In particular, we expect that  we can dissipate the traffic instabilities, and in particular stop-and-go waves, by introducing in the environment a class of vehicles with a faster reaction and with the ability to look further ahead of their current position, meaning by coupling different non-local delayed conservation laws as the one introduced in this paper, see also~\cite{ChiarelloGoatin2019}. Beyond the mathematical interest, the ultimate goal of such modeling framework would be investigating the interaction between human-driven and autonomous vehicles, and proving the stabilization of the traffic flow operated by the latters. This in turn may reduce fuel consumption and pollutant emissions.

\appendix

\section{Technical details} \label{sec:app}

In the following, we give some details about the proof of the limit $\eqref{part5}\rightarrow 0$ for $\Delta x \searrow 0$, which is used in the proof of Theorem \ref{E1}.
Starting from \eqref{perapp}, taking the absolute values and using the bound $\modulo{V^n_{j+1}-V^n_{j-1}}\leq 4\norma{v'}\norma{\omega}R\Delta x$, which is obtained as in \eqref{stimamodulo} for all $n\geq-h$, we get
\begin{align}
\modulo{\eqref{part5}}\leq&\ \frac{1}{2}\Delta t\kappa\sum_{n=1}^{N_T-1}\sum_j\modulo{\left(V^{n-h}_{j+1}-V^{n-h}_{j-1}\right)-\left(V^{n-h-1}_{j+1}-V^{n-h-1}_{j-1}\right)}\modulo{\phi^{n-1}_j}\nonumber\\
&+\frac{1}{2}\Delta t\Delta x\kappa\norma{\partial_x\phi}\sum_{n=1}^{N_T-1}\sum_{j=j_0}^{j_1}\modulo{V^{n-h}_{j+1}-V^{n-h}_{j-1}}\nonumber\\
&+\frac{1}{2}\Delta t\kappa\norma{\phi}\sum_{j=j_0}^{j_1}\modulo{V^0_{j+1}-V^0_{j-1}}\nonumber\\
\leq&\ \frac{1}{2}\Delta t\kappa\norma{\phi}\sum_{n=1}^{N_T-1}\sum_j\modulo{\left(V^{n-h}_{j+1}-V^{n-h}_{j-1}\right)-\left(V^{n-h-1}_{j+1}-V^{n-h-1}_{j-1}\right)}\modulo{\phi^{n-1}_j}\nonumber\\
&+4\kappa\norma{v'}\norma{\omega}RX\left(T\norma{\del_x\phi}\Delta x+\norma{\phi}\Delta t\right)\nonumber\\
=&\ \frac{1}{2}\Delta t\kappa\sum_{n=1}^{N_T-1}\sum_j\modulo{\left(V^{n-h}_{j+1}-V^{n-h}_{j-1}\right)-\left(V^{n-h-1}_{j+1}-V^{n-h-1}_{j-1}\right)}\modulo{\phi^{n-1}_j}\label{tbc}\\
&+C_1\Delta x+C_2\Delta t\nonumber,
\end{align}
where we defined 
\begin{align*}
C_1&=4\kappa\norma{v'}\norma{\omega}RXT\norma{\del_x\phi},\\
C_2&=4\kappa\norma{v'}\norma{\omega}RX\norma{\phi}.
\end{align*}
Taking the absolute values in \eqref{stimavelocita} and using \eqref{stimavelocitaaux}, we get
\begin{align*}
\big|\big(V^{n-h}_{j+1}&-V^{n-h}_{j-1}\big)-\left(V^{n-h-1}_{j+1}-V^{n-h-1}_{j-1}\right)\big|\leq\Delta x\norma{v''} \modulo{\xi^{n-h}_j-\xi^{n-h-1}_j}\norma{\omega}C(T,\norma{\omega},\tau)\tv(\rho^0)\\
&+\Delta x\norma{v'}\Big(\sum_{k=1}^{N}(\omega^{k-1}-\omega^{k+1})\modulo{\rho^{n-h}_{j+k}-\rho^{n-h-1}_{j+k}}\\
&\qquad\qquad\qquad+\omega^0\modulo{\rho^{n-h}_{j-1}-\rho^{n-h-1}_{j-1}}+\omega^1\modulo{\rho^{n-h}_j-\rho^{n-h-1}_j}\Big).
\end{align*}
For some $\theta,\mu\in[0,1]$, we compute
\begin{align*}
\xi^{n-h}_j-\xi^{n-h-1}_j=&\ \Delta x\sum_{k=0}^{+\infty}\left[\mu\omega^k\rho^{n-h}_{j+k+1}+(1-\mu)\omega^k\rho^{n-h}_{j+k-1}-\theta\omega^k\rho^{n-h-1}_{j+k+1}-(1-\theta)\omega^k\rho^{n-h-1}_{j+k-1}\right]\\
=&\ \Delta x\sum_{k=0}^{+\infty}\left[\theta\omega^k\left(\rho^{n-h}_{j+k+1}-\rho^{n-h-1}_{j+k+1}\right)+(1-\theta)\omega^k\left(\rho^{n-h}_{j+k-1}-\rho^{n-h-1}_{j+k-1}\right)\right]\\
&+\Delta x\sum_{k=0}^{+\infty}\left[(\mu-\theta)\omega^k\rho^{n-h}_{j+k+1}+[(1-\mu)-(1-\theta)]\omega^k\rho^{n-h}_{j+k-1}\right]\\
=&\ \Delta x\sum_{k=1}^{N}\left(\theta\omega^{k-1}+(1-\theta)\omega^{k+1}\right)\left(\rho^{n-h}_{j+k}-\rho^{n-h-1}_{j+k}\right)\\
&+\Delta x(\mu-\theta)\left[\sum_{k=1}^{+\infty}(\omega^{k-1}-\omega^{k+1})\rho^{n-h}_{j+k}-\omega^0\rho^{n-h}_{j-1}-\omega^1\rho^{n-h}_j\right]\\
&+\Delta x(1-\theta)\omega^0\left(\rho^{n-h}_{j-1}-\rho^{n-h-1}_{j-1}\right)+\Delta x(1-\theta)\omega^1\left(\rho^{n-h}_j-\rho^{n-h-1}_j\right).
\end{align*}
Since \eqref{delta rho} implies that for every $j\in\mathbb{Z}$
\begin{align}
\modulo{\rho^{n-h}_j-\rho^{n-h-1}_j}\leq&\  \frac{\lambda}{2}\left[\alpha+\left(1+R\norma{f'}\right)V\right]\left(\modulo{\rho^{n-h-1}_{j+1}-\rho^{n-h-1}_j}+\modulo{\rho^{n-h-1}_j-\rho^{n-h-1}_{j-1}}\right)\nonumber\\
&+2R^2\norma{v'}\norma{\omega}\Delta t,\label{utile}
\end{align}
then, similarly to \eqref{delta xi}, we get
\begin{align*}
\big|\xi^{n-h}_j-&\xi^{n-h-1}_j\big|
\leq\ 2\Delta x\norma{\omega}\sum_{k=1}^N\modulo{\rho^{n-h}_{j+k}-\rho^{n-h-1}_{j+k}}+\Delta x R \left(\sum_{k=1}^N(\omega^{k-1}-\omega^{k+1})+4\omega^0\right)\\
\leq&\ \Delta t\norma{\omega}
\left[\alpha+\left(1+R\norma{f'}\right)V\right]
\left(\sum_{k\in\mathbb{Z}}\modulo{\rho^{n-h-1}_{k+1}-\rho^{n-h-1}_k}+\sum_{k\in\mathbb{Z}}\modulo{\rho^{n-h-1}_k-\rho^{n-h-1}_{k-1}}\right)\\
&+4\norma{\omega}R^2\norma{v'}\norma{\omega}\Delta t\sum_{k=1}^N\Delta x+6\norma{\omega}R\Delta x\\
\leq&\ C_3\Delta x+C_4\Delta t,
\end{align*}
with 
\begin{align*}
C_3&=6\norma{\omega}R,\\
C_4&=2\norma{\omega}\left[
\left[\alpha+(1+R\norma{f'})V\right]
C(T,\norma{\omega},\tau)\tv(\rho^0)+2\norma{\omega}R^2\norma{v'}L\right].
\end{align*}
Thus, from \eqref{tbc} it follows
\begin{align*}
\modulo{\eqref{part5}}\leq&\ \frac{1}{2} \Delta t\Delta x\kappa\norma{\phi} \norma{v''}\norma{\omega}C(T,\norma{\omega},\tau)\tv(\rho^0)\sum_{n=1}^{N_T-1}\sum_{j=j_0}^{j_1}\modulo{\xi^{n-h}_j-\xi^{n-h-1}_j}\\
&+\frac{1}{2}\Delta t\Delta x\kappa \norma{v'}\sum_{k=1}^{N}(\omega^{k-1}-\omega^{k+1})\sum_{n=1}^{N_T-1}\sum_j\modulo{\rho^{n-h}_{j+k}-\rho^{n-h-1}_{j+k}}\modulo{\phi^{n-1}_j}\\
&+\frac{1}{2}\Delta t\Delta x\kappa\norma{v'} \sum_{n=1}^{N_T-1}\sum_j\left(\omega^0\modulo{\rho^{n-h}_{j-1}-\rho^{n-h-1}_{j-1}}+\omega^1\modulo{\rho^{n-h}_j-\rho^{n-h-1}_j}\right)\modulo{\phi^{n-1}_j}\\
&+C_1\Delta x+C_2\Delta t\\
\leq&\ \frac{1}{2}\Delta t\Delta x\kappa \norma{v'}\sum_{k=1}^{N}(\omega^{k-1}-\omega^{k+1})\sum_{n=1}^{N_T-1}\sum_j\modulo{\rho^{n-h}_j-\rho^{n-h-1}_j}\modulo{\phi^{n-1}_{j-k}}\\
&+\frac{1}{2}\Delta t\Delta x\kappa\norma{v'} \omega^0\sum_{n=1}^{N_T-1}\sum_j\modulo{\rho^{n-h}_j-\rho^{n-h-1}_j}\left(\modulo{\phi^{n-1}_{j+1}}+\modulo{\phi^{n-1}_j}\right)\\
&+C_5\Delta x+C_6\Delta t\\
\leq&\ 2\Delta t\Delta x\kappa\norma{\phi}\norma{v'}\norma{\omega}\sum_{n=1}^{N_T-1}\sum_{j=j_0-N}^{j_1+N}\modulo{\rho^{n-h}_j-\rho^{n-h-1}_j}\\
&+C_5\Delta x+C_6\Delta t
\end{align*}
being 
\begin{align*}
C_5&=C_1+C_3XT\kappa\norma{\phi}\norma{v''}\norma{\omega}C(T,\norma{\omega},\tau)\tv(\rho^0)\\
C_6&=C_2+C_4XT\kappa\norma{\phi}\norma{v''}\norma{\omega}C(T,\norma{\omega},\tau)\tv(\rho^0).
\end{align*}
Now, applying again \eqref{utile}, one can bound
\begin{align*}
    \modulo{\eqref{part5}}\leq&\ \lambda\left[\alpha+\left(1+R\norma{f'}\right)V\right]\Delta t\Delta x\kappa\norma{\phi}\norma{v'}\norma{\omega}\\
    &\qquad\cdot\sum_{n=1}^{N_T-1}\sum_{j=j_0-N}^{j_1+N}\left(\modulo{\rho^{n-h-1}_{j+1}-\rho^{n-h-1}_j}+\modulo{\rho^{n-h-1}_j-\rho^{n-h-1}_{j-1}}\right)\\
    &+8(X+L)T\kappa\norma{\phi}\norma{v'}^2\norma{\omega}^2R^2\Delta t+C_5\Delta x+C_6\Delta t\\
    \leq&\ \lambda\left[\alpha+\left(1+R\norma{f'}\right)V\right]\kappa\norma{\phi}\norma{v'}\norma{\omega}\\
    &\qquad\cdot\left(\int_0^T\int_{-(X+L)}^{X+L}\modulo{\rho^{\Delta x}\left(t-(h+1)\Delta t,x+\Delta x\right)-\rho^{\Delta x}\left(t-(h+1)\Delta t,x\right)}\d x\d t\right.\\
    &\qquad+\left.\int_0^T\int_{-(X+L)}^{X+L}\modulo{\rho^{\Delta x}\left(t-(h+1)\Delta t,x\right)-\rho^{\Delta x}\left(t-(h+1)\Delta t,x-\Delta x\right)}\d x\d t\right)\\
    &+C_5\Delta x+C_7\Delta t\\
    \leq&\ 2\left[\alpha+\left(1+R\norma{f'}\right)V\right]\kappa\norma{\phi}\norma{v'}\norma{\omega}\mathcal{C}\tv(\rho^0)\Delta t+C_5\Delta x+C_7\Delta t\\
    \leq&\ C_5\Delta x+C_8\Delta t\\
\end{align*}
where the positive constant $\mathcal{C}$ is given by Proposition \ref{BVteo} and
\begin{align*}
C_7&=C_6+8(X+L)T\kappa\norma{\phi}\norma{v'}^2\norma{\omega}^2R^2,\\
C_8&=C_7+2\left[\alpha+\left(1+R\norma{f'}\right)V\right]\kappa\norma{\phi}\norma{v'}
\norma{\omega}\mathcal{C}\tv(\rho^0).
\end{align*}
This proves that \eqref{part5} converges to zero as $\Delta x\rightarrow 0$ ( and $\Delta t\rightarrow 0$).


\section*{Acknowledgments}

This work was funded by the European Union’s Horizon Europe research and innovation programme under the Marie Skłodowska-Curie Doctoral Network Datahyking (Grant No. 101072546).
G. Puppo was also supported by the European Union-NextGenerationEU (National Sustainable Mobility Center CN00000023, Italian Ministry of University and Research Decree n. 1033- 17/06/2022, Spoke 9).

{ \small
	\bibliography{nonlocal}

\def\ocirc#1{\ifmmode\setbox0=\hbox{$#1$}\dimen0=\ht0 \advance\dimen0
  by1pt\rlap{\hbox to\wd0{\hss\raise\dimen0
  \hbox{\hskip.2em$\scriptscriptstyle\circ$}\hss}}#1\else {\accent"17 #1}\fi}
\begin{thebibliography}{10}

\bibitem{AmorimColomboTexeira}
P.~Amorim, R.~Colombo, and A.~Teixeira.
\newblock On the numerical integration of scalar nonlocal conservation laws.
\newblock {\em ESAIM M2AN}, 49(1):19--37, 2015.

\bibitem{Betancourt2011}
F.~Betancourt, R.~B{\"u}rger, K.~H. Karlsen, and E.~M. Tory.
\newblock On nonlocal conservation laws modelling sedimentation.
\newblock {\em Nonlinearity}, 24(3):855--885, 2011.

\bibitem{BlandinGoatin2016}
S.~Blandin and P.~Goatin.
\newblock Well-posedness of a conservation law with non-local flux arising in
  traffic flow modeling.
\newblock {\em Numer. Math.}, 132(2):217--241, 2016.

\bibitem{Burger2018}
M.~Burger, S.~Göttlich, and T.~Jung.
\newblock {Derivation of a first order traffic flow model of
  Lighthill-Whitham-Richards type}.
\newblock {\em IFAC-PapersOnLine}, 51(9):49--54, 2018.
\newblock 15th IFAC Symposium on Control in Transportation Systems CTS 2018.

\bibitem{chiarello2023existence}
F.~A. Chiarello, H.~D. Contreras, and L.~M. Villada.
\newblock Existence of entropy weak solutions for 1d non-local traffic models
  with space-discontinuous flux, 2023.

\bibitem{Chiarello2020}
F.~A. Chiarello, J.~Friedrich, P.~Goatin, and S.~G\"{o}ttlich.
\newblock Micro-macro limit of a nonlocal generalized {A}w-{R}ascle type model.
\newblock {\em SIAM J. Appl. Math.}, 80(4):1841--1861, 2020.

\bibitem{ChiarelloGoatin2018}
F.~A. Chiarello and P.~Goatin.
\newblock Global entropy weak solutions for general non-local traffic flow
  models with anisotropic kernel.
\newblock {\em ESAIM Math. Model. Numer. Anal.}, 52(1):163--180, 2018.

\bibitem{ChiarelloGoatin2019}
F.~A. Chiarello and P.~Goatin.
\newblock Non-local multi-class traffic flow models.
\newblock {\em Netw. Heterog. Media}, 14(2):371--387, 2019.

\bibitem{ChiarelloGoatinRossi2019}
F.~A. Chiarello, P.~Goatin, and E.~Rossi.
\newblock Stability estimates for non-local scalar conservation laws.
\newblock {\em Nonlinear Anal. Real World Appl.}, 45:668--687, 2019.

\bibitem{Du2023}
Q.~Du, K.~Huang, J.~Scott, and W.~Shen.
\newblock A space-time nonlocal traffic flow model: Relaxation representation
  and local limit.
\newblock {\em Discrete and Continuous Dynamical Systems}, 43(9):3456–3484,
  2023.

\bibitem{2018Gottlich}
J.~Friedrich, O.~Kolb, and S.~G\"{o}ttlich.
\newblock A {G}odunov type scheme for a class of {LWR} traffic flow models with
  non-local flux.
\newblock {\em Netw. Heterog. Media}, 13(4):531--547, 2018.

\bibitem{Gottlich2021}
S.~Göttlich, E.~Iacomini, and T.~Jung.
\newblock {Properties of the LWR model with time delay}.
\newblock {\em Networks and Heterogeneous Media}, 16(1):31--47, 2021.

\bibitem{Hilliges1995}
M.~Hilliges and W.~Weidlich.
\newblock A phenomenological model for dynamic traffic flow in networks.
\newblock {\em Transportation Research Part B: Methodological}, 29(6):407--431,
  1995.

\bibitem{KeimerPflug2019}
A.~Keimer and L.~Pflug.
\newblock Nonlocal conservation laws with time delay.
\newblock {\em NoDEA Nonlinear Differential Equations Appl.}, 26(6):Paper No.
  54, 34, 2019.

\bibitem{Kruzkov}
S.~N. Kru{\v{z}}kov.
\newblock First order quasilinear equations with several independent variables.
\newblock {\em Mat. Sb. (N.S.)}, 81 (123):228--255, 1970.

\bibitem{Nelson2000}
P.~Nelson.
\newblock {Synchronized traffic flow from a modified Lighthill-Whitham model}.
\newblock {\em Phys. Rev. E}, 61:R6052--R6055, Jun 2000.

\bibitem{Ngoduy2014}
D.~Ngoduy.
\newblock {Generalized macroscopic traffic model with time delay}.
\newblock {\em Nonlinear Dynamics}, 77, 07 2014.

\bibitem{SIPAHI2007}
R.~Sipahi and S.-I. Niculescu.
\newblock A survey of deterministic time delay traffic flow models.
\newblock {\em IFAC Proceedings Volumes}, 40(23):111--116, 2007.
\newblock 7th IFAC Workshop on Time Delay Systems TDS 2007, Nantes, France,
  17–19 September, 2007.

\bibitem{SopasakisKatsoulakis2006}
A.~Sopasakis and M.~A. Katsoulakis.
\newblock Stochastic modeling and simulation of traffic flow: asymmetric single
  exclusion process with {A}rrhenius look-ahead dynamics.
\newblock {\em SIAM J. Appl. Math.}, 66(3):921--944 (electronic), 2006.

\bibitem{Tordeux2018}
A.~Tordeux, G.~Costeseque, M.~Herty, and A.~Seyfried.
\newblock From traffic and pedestrian follow-the-leader models with reaction
  time to first order convection-diffusion flow models.
\newblock {\em SIAM J. Appl. Math.}, 78(1):63--79, 2018.

\end{thebibliography}
	\bibliographystyle{abbrv}
}

\end{document}